\documentclass[final,leqno]{siamltex}
\usepackage{threeparttable}
\usepackage{dcolumn}
\usepackage{subfigure}
\usepackage[english]{babel}
\usepackage{amssymb}
\usepackage{algorithm}
\usepackage{algorithmic}
\usepackage{color}
\usepackage{graphicx}
\usepackage{booktabs}
\usepackage{amsfonts}
\usepackage{mathrsfs}
\usepackage{amsmath, amssymb}
\usepackage{float}
\usepackage{threeparttable}
\usepackage{dcolumn}
\newcolumntype{d}{D{.}{.}{-1}}
\newtheorem{remark}[theorem]{Remark}
\usepackage{enumitem}

\title{Structure-Preserving Model Reduction of Forced Hamiltonian Systems\thanks{
The authors acknowledge partial support from the Air Force Office of Scientific Research (AFOSR).}}

\author{Liqian Peng\thanks{Department of Mechanical and
Aerospace Engineering, University of Florida, Gainesville, FL 32611-6250
(liqianpeng@ufl.edu).}
        \and Kamran Mohseni\thanks{Department of Mechanical and
Aerospace Engineering, and Department of Electrical and Computer
Engineering, University of Florida, Gainesville, FL 32611-6250
(mohseni@ufl.edu).}}

\begin{document}

\maketitle
\begin{abstract}
This paper reports a development in the proper symplectic decomposition (PSD) for model reduction of forced Hamiltonian systems.  As an analogy to the  proper orthogonal decomposition (POD), PSD is designed to build a symplectic subspace to fit empirical data.   Our aim is two-fold. First, to achieve computational savings for large-scale Hamiltonian systems with  external forces. Second, to simultaneously preserve the symplectic structure and the forced structure of the original system. We first reformulate d'Alembert's principle in the Hamiltonian form. Corresponding to the integral and local forms of d'Alembert's principle, we propose  two different structure-preserving model reduction approaches to reconstruct low-dimensional systems, based on the variational principle and on the structure-preserving projection, respectively. These two approaches are proven to yield the same reduced system. Moreover, by incorporating the vector field into the data ensemble, we provided several algorithms for energy preservation. In a special case when the external force is described by the Rayleigh dissipative function, the proposed method  automatically preserves the dissipativity, boundedness, and stability of the original system. The stability, accuracy, and efficiency of the proposed method are illustrated through numerical simulations of a dissipative  wave equation.
\end{abstract}

\begin{keywords}
Structure-preserving, forced Hamiltonian systems, proper symplectic decomposition, d'Alembert's principle, variational principle, structure-preserving projection,  dissipativity preservation
\end{keywords}

\begin{AMS}
65P10, 37M15, 34C20, 93A15, 37J25
\end{AMS}

\pagestyle{myheadings} \thispagestyle{plain} \markboth{LIQIAN PENG AND KAMRAN MOHSENI}{STRUCTURE-PRESERVING MODEL REDUCTION}

\section{Introduction}\label{sec:intro}
For several centuries, physical models have been used in the natural sciences to describe and predict the world we live in. Physical models are anchored in venerated physical laws, such as Newton's laws of motion, Hamilton's principle, and conservation laws, to name but a few. Compared with data-based empirical models, physical models are  more comprehensive and interpretable.
In most cases, physical models derived from these laws are simple in the sense that they are typically expressed in terms of a few  elegant equations.
Nevertheless, for many practical problems, physical models become computational expensive, and even intractable, when they have high dimensions.



In recent years, the revolution in data sciences has opened a new window for understanding our would. Rather than discovering new physical laws,  empirical models have led to powerful tools for extracting patterns and trends from the data directly.
Although empirical models  are phenomenological, they are predictive as well. When facing the curse of dimensionality in empirical models,
dimensionality reduction techniques in data sciences provide low-cost solutions by pre-processing the data into a lower-dimensional form.

Can we couple tools from the data sciences--tools that are capable of dealing with high dimensions--with physical models? How can we combine the advantages of physical models with the information contained in the data? \emph{Model reduction} is a technique for reducing the computational complexity of physical models in numerical simulations.
Using empirical data, model reduction can provide low-dimensional models that adequately predict the dynamics and allow for real-time analysis and control. The need for model reduction arises because, in many cases, direct numerical simulations are so computationally intensive that they either cannot be performed as often as needed or are only performed in special circumstances.  See~\cite{SorensenDC:01a} for a survey on the classical model reduction methods.

Among these methods, the proper orthogonal decomposition  (POD, also known as Karhunen-Lo\`{e}ve decomposition or principle component analysis) with Galerkin projection, which was first introduced by Moore~\cite{MooreBC:81a}, has wide applications in many fields of  science and engineering.  As an empirical model reduction technique, the POD-Galerkin method (or POD for short) involves an offline-online splitting methodology.  In the offline stage, empirical data is generated by experiments or direct numerical simulations. A reduced model (or reduced system) is then constructed by projecting the full model to a subspace where empirical data approximately resides.  In the online stage, the reduced model is solved in the low-dimensional subspace.

However, the classical POD method is not guaranteed to yield a stable reduced model in general, even if the full model is stable~\cite{PetzoldLR:03a,PrajnaS:03a, Mohseni:14b}. The instability of a reduced model is often accompanied by blowup of system energy.
Thus, a POD reduced model often fails to represent a physical system even if it is conservative or dissipative. More generally, POD can always yield a reduced model with a significantly lower dimension, but the reduced model might be merely a \emph{numerical model}, rather than a \emph{physical model}, since a POD reduced model may not yield to the underlying physical law that exists in the full model. Our primary motivation in this paper is to develop a model reduction technique such that the reduced model is guaranteed to be physical and as stable as the full model.

In the context of classical mechanics, d'Alembert's principle is the fundamental law of motion. Thus, if a reduced model remains physical, it should respect d'Alembert's principle. In section \ref{sec:forced}, we shall see that forced Hamiltonian equations satisfy d'Alembert's principle; conversely,  if a system satisfies d'Alembert's principle, it  can be represented by a forced Hamiltonian equation when choosing canonical coordinates. Thus, this paper focuses on developing a structure-preserving  model reduction method for forced Hamiltonian systems, where the structure refers to the forced Hamiltonian structure and the systems are represented by ordinary differential equations (ODEs).

The proposed method in this paper extends our previous work on the symplectic model reduction of Hamiltonian systems to the structure-preserving model reduction of forced and dissipative systems. The symplectic model reduction is based on proper symplectic decomposition (PSD)-symplectic projection method~\cite{Mohseni:14b}. Analogous to the POD-Galerkin method, PSD builds a symplectic subspace to fit empirical data, while the symplectic projection constructs a reduced Hamiltonian system on the symplectic subspace. Because the PSD reduced system preserves the symplectic structure, it automatically preserves the system energy and stability.   Owing to these properties, the symplectic method outperforms  the POD for model reduction of Hamiltonian systems, especially when stability is taken into consideration for long-time integration. Since many physical and engineering systems have external forces, this paper applies symplectic algorithms for model reduction of more general dynamical systems with external forces. Besides PSD, there are also other structure-preserving model reduction methods in the context of classical mechanics, including the Lagrangian approach~\cite{Marsden:03h, CarlbergK:14a} and the port-Hamiltonian approach~\cite{HartmannC:10a,vdSchaft:10a,vdSchaft:12a}. Compared with  these methods, PSD is directly related to symplectic geometry, and provides more flexibility to construct an optimal subspace to fit empirical data.

The remainder of this paper is organized as follows.  Since forced Hamiltonian equations are anchored in d'Alembert's principle,  we reformulate d'Alembert's principle in the Hamiltonian form in section \ref{sec:forced}. Corresponding to the integral and local forms of d'Alembert's principle, section \ref{sec:twoapp} presents  two different structure-preserving  approaches for model reduction of forced Hamiltonian equations based on the variational principle and on the structure-preserving projection, respectively.  We also prove that the two approaches are equal in the sense that they yield the same reduced equation and provide a PSD algorithm to construct the reduced basis function. In section \ref{sec:dissipative}, we discuss dissipative Hamiltonian systems, and prove that the proposed method  automatically preserves the dissipativity. In section \ref{sec:numerical_diss}, the stability, accuracy, and efficiency of the proposed technique are illustrated through numerical simulations of a dissipative  wave equation.  Finally, conclusions are offered in section \ref{sec:conclusion}.

\section{Forced Hamiltonian equations}\label{sec:forced} In Hamiltonian mechanics, a mechanical system with external forces can be represented by a forced Hamiltonian equation. In this section, we first represent the forced Hamiltonian equation, and then derive it by d'Alembert's principle.

Let $Q$ be an $n$-dimensional vector space over $\mathbb{R}$, $Q^*$ be its dual space, and ${\langle  \cdot, \cdot  \rangle}:Q^*\times Q \to \mathbb{R}$   be a nondegenerate duality paring. With $\mathbb{V}=Q  \oplus Q^*$, the pair $(\mathbb{V}, \Omega)$ is a symplectic vector space, where $\mathbb{V}$ is the phase space and $\Omega$ is a closed non-degenerate two-form on $\mathbb{V}$.     Assigning a symplectic form $\Omega$ to $\mathbb{V}$ is referred to as giving $\mathbb{V}$ a symplectic structure. With $u_q, v_q \in Q$ and  $u_p, v_p \in Q^*$, we have
 \begin{equation}
 \Omega((u_q, u_p), (v_q, v_p)) = \langle v_p,  u_q\rangle - \langle u_p, v_p\rangle.
 \end{equation}
 Using canonical coordinates,  $\Omega$ is represented by  the Poisson matrix
\begin{equation*}
 \qquad {J_{2n}} = \begin{bmatrix}
  0_n & {I_n} \\
 {-I_n} & 0_n \\
 \end{bmatrix}.
\end{equation*}

Let $H:\mathbb{V} \to \mathbb{R}$ denote a smooth Hamiltonian function. The time evolution of forced Hamiltonian systems are defined by
\begin{equation}\label{diffhamiltons}
\dot q=  \nabla_p  H(q,p), \qquad
\dot p= -\nabla_q  H(q,p)+f_H(q,p),
 \end{equation}
 where  $q\in Q$ denotes the generalized coordinate, $p \in Q^*$ denotes the generalized momentum, and $f_H(q,p) \in Q^*$ is a force field.  We abstract this formulation by introducing a variable $x=(q,p)$ in the phase space $\mathbb{V}$. Then, (\ref{diffhamiltons}) becomes
\begin{equation}\label{fom0}
\dot x= X_H(x)+X_F(x),
\end{equation}
where $X_H(x)=J_{2n} \nabla_x H(x)$ denotes a Hamiltonian vector field,  and $X_F(x)=(0, f_H(x))$ denotes a vertical vector field with zero in its first component. The state variable $x(t)$ can also be considered as a function of $t$, which gives a trajectory  as $t$ varies over $\mathbb{R}_+$ with a fixed initial condition $x_0$. The trajectory $x(\mathbb{R}_+)$ contains a sequence of states  that follow from $x_0$.

Dissipative Hamiltonian systems are special forced Hamiltonian systems, where the system energy is decreasing with time. As an example, consider a one-dimensional  harmonic oscillator with undamped angular frequency $\omega_0$ and damping ratio $\zeta$. Newton's second law takes the form
\begin{equation}\label{examp}
 \ddot x +2\zeta \omega_0 \dot x +\omega_0^2 x =0.
 \end{equation}
With $q=x$ and $p=\dot x$, the Hamiltonian function is given by $H(q,p)={\frac{1}{2}}p^2 +\frac{1}{2}\omega_0^2q^2$, and the force field is given by  $f_H(q,p)=-2\zeta \omega_0 p$. Plugging $H(q,p)$ and $f_H(q,p)$ into (\ref{diffhamiltons}), we can get the Hamiltonian representation of the harmonic oscillator, which is exactly the same as (\ref{examp}).
The system energy is given by $E(t):=H(q(t), p(t))$. The time derivative of $E(t)$ is given by $\dot E(t)=-2 \zeta \omega_0 p(t)^2$, which is negative for every $t$.

Forced Hamiltonian equations can be derived from the Legendre transformation of Euler--Lagrange equations with generalized forces. Alternatively,  they can be directly obtained from the reformulation of d'Alembert's principle in the Hamiltonian coordinates. In this paper, we take the second approach, since this approach also gives us insight on reconstructing structure-preserving reduced models.  Our derivation closely follows reference~\cite{Marsden:03j} (pp. 205--210), where generalized Euler--Lagrangian equations are obtained from d'Alembert's principle in the Lagrangian coordinates. We shall begin with  the structure of the vertical vector field.

\subsection{Force fields} Let $H: \mathbb{V} \to \mathbb{R}$ be a Hamiltonian function,  $X_H: \mathbb{V} \to \mathbb{V}$ be the Hamiltonian vector field associated to $H$, and $\pi_Q: \mathbb{V} \to Q, (q, p) \mapsto q$ be the canonical projection. A vector field $X_F: \mathbb{V} \to \mathbb{V}$ is called \emph{vertical} if the projection of $X_F$ is zero, i.e., $\pi_Q  (X_F) = 0$. Such a vector field $X_F$ defines a one-form $\Delta^F: \mathbb{V} \to \mathbb{V}^*$ by contraction with $\Omega$: $\Delta^F = -i_{X_F} \Omega$. For any vertical vector field $u$ on $\mathbb{V}$, if the duality paring $\Delta^F \cdot u=0$, then we say  $\Delta^F$ is \emph{horizontal}. Here, we use the dot product to represent the duality paring of $\mathbb{V}^*$ and $\mathbb{V}$.

\medskip
\begin{proposition}\label{7.8.1}
 If $X_F$ is vertical, then the corresponding one-form $\Delta^F$ is horizontal. Conversely, given a horizontal one-form $\Delta^F$ on $\mathbb{V}$, the vector field $X_F$ on $\mathbb{V}$, given by $\Delta^F = -i_{X_F} \Omega$, is vertical.
\end{proposition}

\begin{proof}  Let $X_F=(f_q, f_p)$ and  $u=(u_q, u_p)$, where $f_q,  u_q, \in Q$ and  $f_p, u_p \in Q^*$.  Using the definition of $\Omega$, we have
\begin{equation}
\Delta^F \cdot u=-i_{X_F} \Omega(u)=-\Omega((f_q, f_p),(u_q, u_p)) = -\langle u_p,  f_q\rangle + \langle f_p, u_q\rangle.
\end{equation}
If $u$ is vertical, $u_q=0$. Thus,  $\Delta^F \cdot u = 0$ for every vertical $u$ is equivalent to $\langle u_p,  f_q\rangle=0$ for every $u_p$, which holds if and only if $f_q=0$, i.e., the vector field $X_F$ is vertical.
 \hfill
\end{proof}
\medskip

\begin{proposition}\label{prop:7.8.2}
A \emph{force field} $f_H : \mathbb{V}\to Q^*$  induces a horizontal one-form $\Delta^F$ on $\mathbb{V}$ by
\begin{equation}\label{7.8.2}
\Delta^F \cdot u = \langle {f_H, \pi_Q (u)} \rangle,
\end{equation}
where  $u$ is a vector field on $\mathbb{V}$. Conversely, formula (\ref{7.8.2}) defines a map $f_H$ for any horizontal one-form $\Delta^F$.
\end{proposition}
\medskip

\begin{proof}
Given $f_H$, (\ref{7.8.2}) defines  a smooth one-form $\Delta^F$ on $\mathbb{V}$. If $u$ is vertical, then the right-hand side of  (\ref{7.8.2}) vanishes, and so $\Delta^F$ is horizontal. Conversely, if  $\Delta^F$ is a horizontal one-form on $\mathbb{V}$, $\Delta^F\cdot (u_q, u_p)=\Delta^F\cdot (u_q, 0)+\Delta^F\cdot (0, u_p)=\Delta^F\cdot (u_q, 0)$ for any vector field $u=(u_q, u_p)$.
Thus, (\ref{7.8.2}) is equivalent to  $\langle {f_H, u_q} \rangle=\Delta^F\cdot (u_q, 0)$. Since ${\langle  \cdot, \cdot  \rangle}$   is nondegenerate,  $f_H$ is well-defined and smooth.
 \hfill
\end{proof}
\medskip

Propositions \ref{7.8.1} and \ref{prop:7.8.2} imply that a force field $f_H$ introduces a horizontal one-form $\Delta^F$, which in turn determines a vertical vector field $X_F$. Using canonical coordinates, if $f_H(q,p) \in \mathbb{R}^{n}$ denotes a force field, then a horizontal one-form is given by $\Delta^F(q, p)=(f_H(q, p), 0)$. By contraction with  $\Omega$,  the corresponding vertical field is given by $X_F(q,p)=-J_{2n}\Delta^F(q, p)=(0, f_H(q,p))$.  Treating $\Delta^F(q,p)$ as the external force term on a mechanical system with a Hamiltonian $H$, we will next derive the  equation of motion by d'Alembert's principle.

\subsection{D'Alembert's principle}
D'Alembert's principle is a statement of the fundamental  law of motion in classical mechanics. It  is more general than Hamilton's principle since it considers both internal and  external forces.
In Newton's coordinates, the principle can be written as
\begin{equation}
\Sigma_i(F_i-\dot p_i)\cdot r_i=0,
 \end{equation}
 where $F_i$ is the total applied force (excluding constraint forces) on the $i$-th particle, $p_i$ is the momentum of the $i$-th particle, and $\delta r_i$ is the virtual displacement of the $i$-th particle which is consistent with the constraints. We shall reformulate d'Alembert's principle by Hamiltonian coordinates.


\medskip
\begin{definition}\label{7.8.4}
Given a Hamiltonian function $H(q,p)$ and a horizontal one-form $\Delta^F(q,p)$,  the integral d'Alembert's principle for a trajectory $(q(t), p(t))$ in $\mathbb{V}$ is
\begin{equation}\label{integral}
\delta \int_a^b {L(q(t),p(t))dt}  + \int_a^b \Delta^F(q(t),p(t)) \cdot (\delta q, \delta p) { dt}  = 0,
\end{equation}
where $(\delta q, \delta p)$ is a variation on $\mathbb{V}$, and $L:\mathbb{V}\to \mathbb{R}$ is the Lagrangian function  defined by $L(q,p)=\langle p, \dot q\rangle -H(q,p)$.
\end{definition}
\medskip

The variation of the first term is given by the usual expression
\begin{align*}
 \delta \int_a^b {L(q,p)dt} &=
 \int_a^b {\langle p, \delta \dot q\rangle + \langle  \delta p, \dot q \rangle - (\nabla_q H(q,p), \nabla_p H(q,p)) \cdot (\delta q, \delta p) dt}  \\
 &= \int_a^b {\langle p, \delta \dot q\rangle + \langle  \delta p, \dot q \rangle - \langle \nabla_q H(q,p), \delta q \rangle - \langle \delta p, \nabla_p H(q,p) \rangle dt}  \\
  &= \int_a^b { \langle -\dot p - \nabla_q H(q,p), \delta q \rangle +  \langle \delta p,  \dot q - \nabla _p H(q,p) \rangle dt}
 \end{align*}
for a given variation $(\delta q, \delta p)$, which vanishes at the endpoints. Since  the external force $\Delta^F(q,p)$ is horizontal,  Proposition (\ref{7.8.2}), implies that
\begin{equation*}
\Delta^F(q,p)\cdot (\delta q, \delta p)=\langle f_H(q,p), \delta q \rangle,
\end{equation*}
where $f_H$ is the force field corresponding to $\Delta^F$. Thus,  (\ref{integral}) gives
\begin{equation}\label{integral1}
\int_a^b { \langle -\dot p - \nabla_q H(q,p)+f_H(q,p), \delta q \rangle +  \langle \dot q - \nabla _p H(q,p), \delta p \rangle dt}  = 0.
\end{equation}
Therefore, the trajectory of the integral d'Alembert's principle is given by the forced Hamiltonian equation (\ref{diffhamiltons}).

We can also formulate an equivalent principle in terms of one-forms.

\medskip
\begin{definition}\label{7.8.3}
Given a Hamiltonian function $H(x)$ and a horizontal one-form $\Delta^F(x)$, the local d'Alembert's principle for the ultimate equation of motion,  $\dot x(t) =X(x(t))$,  is determined by
\begin{equation}\label{localdAlm}
i_{X} \Omega(x)=dH(x)-\Delta^F(x),
\end{equation}
where $X$ denotes the forced Hamiltonian vector field on $\mathbb{V}$ .
\end{definition}
\medskip

\begin{proposition}
The two forms of  d'Alembert's principle are equivalent, i.e., they give the same equation of motion.
\end{proposition}
\medskip
\begin{proof}
Let $X_F$ denote the  vector field associated to $\Delta^F$, i.e., $\Delta^F=-i_{X_F} \Omega$. Since $\Delta^F$ is horizontal, $X_F$ is vertical. Since $dH=i_{X_H}\Omega$,
\begin{equation}\label{XXH}
X(x) = X_H(x) + X_F(x)
\end{equation}
satisfies the local d'Alembert's principle. Conversely,  the only vector field $X$ satisfying the local d'Alembert's principle is given by (\ref{XXH}),   and uniqueness is guaranteed by nondegeneracy of $\Omega$. Therefore, both the integral and local forms of d'Alembert's principle give the same vector field for the equation of motion $\dot x(t)=X(x(t))$. \hfill
\end{proof}
\medskip

From now on, we will refer to both (\ref{integral}) and (\ref{localdAlm}) as simply  d'Alembert's principle. By the above analysis, if a system satisfies d'Alembert's principle, the equation of motion is given by the forced Hamiltonian equation. Conversely, if a system is represented by a forced Hamiltonian equation, it automatically satisfies d'Alembert's principle. Since  d'Alembert's principle is the first principle  in classical mechanics, any mechanical  system  can be represented by a forced Hamiltonian equation. If a vector field $X$ can be represented by the sum of a Hamiltonian vector field $X_H$ and a vertical  vector field $X_F$, we say $X$ has forced Hamiltonian structure. In the next section, we develop a new model reduction method  which preserves the forced Hamiltonian structure.


\section{Reduction and reconstruction of dynamics}\label{sec:twoapp}
In this section, we propose two methods to construct reduced dynamics in a low-dimensional subspace. The first approach is based on the variational principle, which is closely related to the integral d'Alembert's principle;  the second approach is based on the structure-preserving projection, which is closely related to the local d'Alembert's principle. Both methods take advantage of empirical data to construct a reduced system, while simultaneously preserving the underlying forced Hamiltonian structure. In other words, if the original system is a forced Hamiltonian equation, the reduced system remains a forced Hamiltonian equation, but with significantly fewer dimensions.

\subsection{Variational principle}\label{subsec:var} In the context of Lagrangian mechanics, the variational principle was used to yield reduced systems while preserving the Lagrangian structure ~\cite{Marsden:03h, CarlbergK:14a}. The idea is to insert $q=\Phi r$ into the Lagrangian  to obtain a reduced system in terms of $r$. Here,  $\Phi\in \mathbb{R}^{n\times k}$ denotes a POD basis matrix and $r\in \mathbb{R}^k$ denotes the reduced coordinates.  Since $\dot q$ is the time derivative of $q$, it is fixed by $\dot q=\Phi \dot r$. The Hamiltonian approach provides more flexibility, since $q$ and $p$ have the same status in the phase space.

Let $(\mathbb{V}, \Omega)$ and $(\mathbb{W}, \omega)$ be two symplectic vector spaces; $\dim(\mathbb{V})=2n$, $\dim(\mathbb{W})=2k$, and $k \le n$.
Using canonical coordinates, a lift $\sigma:\mathbb{V} \to \mathbb{W}, z\mapsto x$  can be written as
\begin{equation}\label{Az}
x=Az,
\end{equation}
where $A\in \mathbb{R}^{2n\times 2k}$. Using the block form, $z=(r,s)$, $x=(q,p)$, and
\begin{equation}\label{matA}
 A=
 \begin{bmatrix}
  A_{qq} & A_{qp} \\
 A_{pq} & A_{pp} \\
 \end{bmatrix}.
\end{equation}
Then,  the map $x=\sigma (z)$ is represented by
\begin{equation*}
  \begin{bmatrix}
 q \\
 p \\
 \end{bmatrix}=
 \begin{bmatrix}
  A_{qq} & A_{qp} \\
 A_{pq} & A_{pp} \\
 \end{bmatrix}
  \begin{bmatrix}
 r \\
 s \\
 \end{bmatrix}=
   \begin{bmatrix}
  A_{qq}r + A_{qp}s \\
 A_{pq}r+ A_{pp}s \\
 \end{bmatrix}.
\end{equation*}
In order to construct a reduced equation, we can plug $q=A_{qq}r + A_{qp}s$ and $p=A_{pq}r + A_{pp}s$ into (\ref{integral}) and take the variation on $(\delta r, \delta s)$. This yields
\begin{equation*}
 M \begin{bmatrix}
 \dot r \\
 \dot s \\
 \end{bmatrix}=
 \begin{bmatrix}
  \nabla_r \tilde H(r,s) \\
 \nabla_s \tilde H(r,s) \\
 \end{bmatrix}
  -
  \begin{bmatrix}
 A_{qq}^T \tilde f_H(r,s) \\
 A_{qp}^T \tilde f_H(r,s) \\
 \end{bmatrix},
\end{equation*}
where $\tilde H(r,s)=H(A_{qq}r + A_{qp}s, A_{pq}r + A_{pp}s)$, $\tilde f_H(r,s)=f_H(A_{qq}r + A_{qp}s, A_{pq}r + A_{pp}s)$, and
\begin{equation*}
 M=
\begin{bmatrix}
 A_{pq}^T A_{qq}-A_{qq}^T A_{pq} & A_{pq}^T A_{qp}-A_{qq}^T A_{pp} \\
 A_{pp}^T A_{qq}-A_{qp}^T A_{pq} & A_{pp}^T A_{qp}-A_{qp}^T A_{pp} \\
 \end{bmatrix}.
\end{equation*}
Suppose $M$ is invertible, we obtain
\begin{equation}\label{varrs1}
\begin{bmatrix}
 \dot r \\
 \dot s \\
 \end{bmatrix}= M^{-1}
 \begin{bmatrix}
  \nabla_r \tilde H(r,s) \\
 \nabla_s \tilde H(r,s) \\
 \end{bmatrix}
  -M^{-1}
  \begin{bmatrix}
 A_{qq}^T \tilde f_H(r,s) \\
 A_{qp}^T \tilde f_H(r,s) \\
 \end{bmatrix}.
\end{equation}
Equation (\ref{varrs1}) is the reduced system constructed by the variational principle.  Next, we add some constraints to $A$ such that   (\ref{varrs1})  preserves the forced Hamiltonian structure.


\medskip
\begin{definition}
Let $(\mathbb{V}, \Omega)$ and $(\mathbb{W}, \omega)$ be two symplectic vector spaces; $\dim(\mathbb{V})=2n$, $\dim(\mathbb{W})=2k$, and $k \le n$.
A lift $\sigma: \mathbb{W} \to \mathbb{V}$ is called symplectic if it preserves the symplectic structure:
\begin{equation} \label{symp}
 \omega(z, w)=\Omega(\sigma (z), \sigma (w)),
\end{equation}
for every  $z, w \in \mathbb{W}$.
\end{definition}
\medskip

Let $A$ denote the matrix form of $\sigma$ in canonical coordinates, then (\ref{symp}) implies
\begin{equation}\label{AJA1}
A^T J_{2n} A =J_{2k}.
\end{equation}
In this case, we say the matrix $A$ is symplectic, written as $A\in Sp({2k},\mathbb{R}^{2n})$, where
\begin{equation}
Sp({2k},\mathbb{R}^{2n}):=\{A\in \mathbb{R}^{2n\times 2k}| A^T J_{2n} A= J_{2k}\}
\end{equation}
denotes the \emph{symplectic Stiefel manifold}.

 The symplectic condition  can also be represented in  block form by plugging (\ref{matA})  into
(\ref{AJA1}).

\medskip
\begin{proposition}\label{4.1}
 The matrix $A=[A_{qq}, A_{qp}; A_{pq}, A_{pp}]$  is symplectic if and only if $A_{qq}^TA_{pq}$ and $A_{qp}^TA_{pp}$ are symmetric and $A_{qq}^TA_{pp}-A_{pq}^TA_{qp}=I_k$.
\end{proposition}
\medskip

The next lemma gives a sufficient and necessary condition such that the variational  approach is structure-preserving for any forced Hamiltonian equations.

\medskip
\begin{lemma}
The reduced equation (\ref{varrs1}) constructed by the variational principle preserves the forced Hamiltonian structure for any Hamiltonian functions $H(q,p)$ and force fields $f_H(q,p)$  if and only if $A\in Sp({2k},\mathbb{R}^{2n})$ and $A_{qp}=0$.
\end{lemma}
\medskip
\begin{proof}
If $A\in Sp({2k},\mathbb{R}^{2n})$ and $A_{qp}=0$, then (\ref{varrs1}) reduces to
\begin{equation}\label{reduceddiff}
\begin{bmatrix}
 \dot r \\
 \dot s \\
 \end{bmatrix}= J_{2n}
 \begin{bmatrix}
  \nabla_r \tilde H(r,s) \\
 \nabla_s \tilde H(r,s) \\
 \end{bmatrix}
  +
  \begin{bmatrix}
 0 \\
A_{qq}^T \tilde f_H(r,s) \\
 \end{bmatrix}.
\end{equation}
where $\tilde H(r,s)$ represents the reduced Hamiltonian function and $A_{qq}^T \tilde f_H(r,s)$ represents the reduced force field.

Conversely, suppose that the reduced equation (\ref{varrs1}) preserves the forced Hamiltonian structure for any high-dimensional systems of the form (\ref{diffhamiltons}). Let $f_H(q,p)=0$, then (\ref{varrs1}) reduces to
\begin{equation}\label{varrs2}
\begin{bmatrix}
 \dot r \\
 \dot s \\
 \end{bmatrix}= M^{-1}
 \begin{bmatrix}
  \nabla_r \tilde H(r,s) \\
 \nabla_s \tilde H(r,s) \\
 \end{bmatrix}.
\end{equation}
If this equation is Hamiltonian for any $\tilde H(r,s)$, we must have  $M^{-1}=J_{2k}$. By Proposition \ref{4.1},  $M^{-1}=J_{2k}$ is equivalent to $A\in Sp({2k},\mathbb{R}^{2n})$. Now we plug  $M^{-1}=J_{2k}$ into (\ref{varrs1}). If the second term on the right-hand side of  (\ref{varrs1}) is a vertical, then  $ A_{qp}^T \tilde f_H(q,p)=0$. Since $\tilde f_H(q,p)$ can be arbitrary, this implies that $A_{qp}=0$.
\hfill
\end{proof}

\subsection{Structure-preserving  projection}\label{subsec:projection}
In~\cite{Mohseni:14b}, the symplectic  projection was proposed to construct reduced models for  Hamiltonian equations while preserving the symplectic structure.
In this section, we extend the symplectic projection to structure-preserving projection of forced Hamiltonian systems.
The idea is to add some constraints to the symplectic projection so that the new projection also preserves the structure of the vertical vector field.

\subsubsection{Symplectic projection} We begin with the basic definition of the symplectic projection.

\begin{definition}
Suppose $\sigma:\mathbb{W} \to \mathbb{V}$ is a symplectic lift. Then the adjoint of $\sigma$ is the linear mapping $\pi:  \mathbb{V}  \to \mathbb{W}$ satisfying
 \begin{equation}\label{projection000}
\omega(w, \pi(x))=\Omega(\sigma(w),x),
\end{equation}
for every $w\in \mathbb{W}$ and $x\in \mathbb{V}$. We say $\pi$ is the symplectic projection induced by $\sigma$.
\end{definition}
\medskip

Using canonical coordinates, $\sigma$ can be represented by a symplectic matrix $A$. Then, the symplectic projection $\pi: x \mapsto z$  can be written as
 \vspace{-2 mm}
 \begin{equation}
z=A^+x,
\end{equation}
where $A^+\in \mathbb{R}^{2k\times 2n}$. Equation (\ref{projection000}) implies that $J_{2k}A^+=A^TJ_{2n}$. Since $J_{2k}$ is invertible, it follows that
\begin{equation}\label{A+}
A^+=J_{2k}^TA^TJ_{2n}.
\end{equation}

Since $A^TJ_{2n}A=J_{2k}$,  $A^+$ is a left inverse of $A$, i.e.,
\begin{equation}\label{AA+=I}
A^+A=I_{2k}.
\end{equation}
In general, $A^+$ is not equal to the Moore--Penrose pseudoinverse $(A^TA)^{-1}A^T$, and the left inverse of $A$ is not unique. However, since $\Omega$ and $\omega$ are nondegenerate, $A^+$ is the unique  adjoint matrix of $A$ with respect to the Poisson matrices $J_{2n}$ and $J_{2k}$.

Equation (\ref{AA+=I}) implies that  $\pi\circ \sigma=id_{\mathbb{W}}$. Since $(\sigma \circ \pi) \circ (\sigma \circ \pi)=\sigma \circ (\pi \circ \sigma) \circ \pi=\sigma \circ \pi$, $\sigma \circ \pi$ defines  a projection operator on $\mathbb{V}$.
\medskip

\begin{proposition}\label{lem:3.7}
Suppose $\sigma:\mathbb{W} \to \mathbb{V}$ is a symplectic lift and $\pi:  \mathbb{V}  \to \mathbb{W}$ is a symplectic projection  introduced by $\sigma$. Then
\begin{equation}\label{constantiden}
\Omega(u, (\sigma \circ \pi)(v))=\Omega((\sigma \circ \pi)(u),v)=\Omega((\sigma \circ \pi)(u),(\sigma \circ \pi)(v))=\omega(\pi(u), \pi(v)),
\end{equation}
for every $u, v\in \mathbb{V}$.
\end{proposition}
\medskip

\begin{proof}
Using canonical coordinates, (\ref{constantiden}) can be rewritten as
\begin{equation}\label{constantiden1}
u^T J_{2n} AA^+v=(AA^+u)^T J_{2n} v=(AA^+u)^T J_{2n} (AA^+v)=(A^+u)^T J_{2k} (A^+v),
\end{equation}
which can be  verified by replacing $A^+$ with (\ref{A+}).
\hfill
\end{proof}
\medskip

Since $\Omega$ is skew-symmetric,  Proposition \ref{lem:3.7} implies that
\begin{equation*}
\Omega(u, (\sigma \circ \pi)(u))=\Omega((\sigma \circ \pi)(u),u)=0.
\end{equation*}
As a consequence, for every $u,v \in \mathbb{V}$,
\begin{equation}\label{zerocond}
\Omega(u, (\sigma \circ \pi)(u+v))=\Omega(u, (\sigma \circ \pi)(v)).
\end{equation}

The symplectic projection defines a mapping from a high-dimensional space to a low-dimensional space. The same projection can also be applied a high-dimensional Hamiltonian system to obtain a reduced system while preserving the symplectic structure.

To see this, suppose the original system is Hamiltonian, i.e., $\dot x= X_H(x)$. Suppose  $x=Az$. Using the chain rule, we obtain ${\nabla _z}H(Az)=A^T{\nabla _x}H(x)$. Using $A^+J_{2n}=J_{2k}A^T$, we obtain the symplecitc projection of the tangent vector at $x$,
\begin{equation}
\pi (X_H(x))=A^+J_{2n}{\nabla _x}H(x) = J_{2k}A^T {\nabla _x}H(x)=J_{2k}{\nabla_z}H(Az)=J_{2k}{\nabla_z}\tilde H(z),
\end{equation}
where $\tilde H(z):=H(Az)$ defines a Hamiltonian function on $\mathbb{W}$. Since $X_{\tilde H}(z):=J_{2k}{\nabla_z}\tilde H(z)$  gives a Hamiltonian vector filed on $\mathbb{W}$, the reduced system $\dot z= X_{\tilde H}(z)$ is a well-defined and preserves the symplectic structure.

 With some extra constraints,  the next section shows that the symplectic projection  can also be applied to a forced Hamiltonian system to construct a reduced system while preserving the structure of the vertical vector field.

\subsubsection{Structure-preserving projection}
For a forced Hamiltonian system,  the corresponding vector field is given by $X(x)=X_H(x)+X_F(x)$ at each $x\in \mathbb{V}$. Then, we can define a reduced vector field by
\begin{equation}
\pi (X(\sigma(z)))=\pi (X_H(\sigma(z)))+\pi (X_F(\sigma(z))),
 \end{equation}
 at each $z\in \mathbb{W}$. If $\sigma$ and $\pi$ respectively represent the symplectic lift and symplectic projection,  the last section shows that $\pi (X_H(\sigma(z)))$ gives a  Hamiltonian vector field on $\mathbb{W}$. Thus, if the reduced system preserves the forced Hamiltonian structure, we only need $\pi (X_F(\sigma(z)))$ to be a vertical vector field on $\mathbb{W}$.

In block form, $A^+$ can be written as
\begin{equation}
 A^+=
\begin{bmatrix}
  A_{pp}^T & -A_{qp}^T \\
 -A_{pq}^T & A_{qq}^T \\
\end{bmatrix}.
\end{equation}
It follows that the projection $X_{\tilde F}(z)$ of the vertical vector field $X_{F}$ at $Az$ is given by
\begin{equation}
X_{\tilde F}(z)=A^+ X_F(Az)=
\begin{bmatrix}
  A_{pp}^T & -A_{qp}^T \\
 -A_{pq}^T & A_{qq}^T \\
\end{bmatrix}
\begin{bmatrix}
 0 \\
 f_H(Az) \\
\end{bmatrix}
=
\begin{bmatrix}
 -A_{qp}^T f_H(Az)  \\
 A_{qq}^T f_H(Az)  \\
\end{bmatrix}.
\end{equation}
Thus, the symplectic projection preserves the forced structure, i.e.,  $\pi X_F(Az)$ is vertical if and only if  $A_{qp}^T f_H(Az)=0$ for any $f_H(Az)$. This is equivalent to $A_{qp}=0$.

  \medskip
\begin{definition}
Let $z\in \mathbb{W}$ and $x\in \mathbb{V}$.
Using the canonical coordinates, a linear mapping $\pi: x  \mapsto z$ is a \emph{structure-preserving projection} if there exists a symplectic matrix $A\in Sp({2k},\mathbb{R}^{2n})$ with $A_{qp}=0$, such that
 \begin{equation}
z=A^+x.
\end{equation}
\end{definition}

 Now, suppose $A\in Sp({2k},\mathbb{R}^{2n})$, $A_{qp}=0$, and $x(t)\in {\rm{Range}}(A)$ for every $t$. Then, $x(t)=Az(t)$.
 Taking the time derivative of $z=A^+x$ and using (\ref{fom0}), the time evolution of $z(t)$ is given by
\begin{equation}\label{symplecticpro}
\dot z = \pi\dot x = \pi X_H(x) +\pi X_F(x) = X_{\tilde H}(z) +  X_{\tilde F}(z).
\end{equation}
Even if   $x(t)  \notin  {\rm{Range}}(A)$ for some $t$, the last expression is still well-defined forced Hamiltonian vector field. Thus, the reduced system constructed by the structure-preserving projection preserves the  forced Hamiltonian structure.

 \medskip

\begin{remark}\label{rmk:3.8}
\emph{Both the variational principle and the structure-preserving projection methods requires that $A\in Sp({2k},\mathbb{R}^{2n})$ and $A_{qp}=0$, which means that the two methods can share the same basis matrix $A$. Let $z=(r,s)$, then (\ref{reduceddiff}) and (\ref{symplecticpro}) define the same system on the subspace spanned by the column vectors of $A$.  Thus, two  methods  construct the same reduced system. From now on, we do not distinguish the variational principle and the structure-preserving projection when we mention a structure-preserving reduced model.
}
\end{remark}

\medskip

 \medskip
\begin{definition}
Given a $2n$-dimensional forced Hamiltonian system (\ref{fom0}) with an initial condition  $x_0\in \mathbb{R}^{2n}$, the structure-preserving reduced model  is a $2k$-dimensional ($k\le n$) system
 \begin{equation}\label{symprodiff}
\dot z =  X_{\tilde H}(z) +  X_{\tilde F}(z),
\end{equation}
with the initial condition  $z_0=A^+ x_0 \in \mathbb{R}^{2k}$, where $A\in Sp({2k},\mathbb{R}^{2n})$, $A_{qp}=0$, $X_{\tilde H}(z)$ and $X_{\tilde F}(z)$ respectively represent the  Hamiltonian vector field and the vertical vector  on $\mathbb{R}^{2k}$.
\end{definition}

\medskip

\begin{remark}\emph{
The vertical vector field $X_{\tilde F}$ defines a horizonal one-form $\Delta^{\tilde F}$  by contraction with $\omega$, i.e., $\Delta^{\tilde F} = -i_{X_{\tilde F}} \omega$. It follows that
\begin{equation}\label{redlocaldAlm}
i_Z \omega=d {\tilde H}- \Delta^{\tilde F}.
\end{equation}
This verifies that  the reduced system constructed by the structure-preserving projection also satisfies the local d'Alembert's  principle.}
\end{remark}
\medskip

\begin{remark}
\emph{
Suppose $\pi$ is a structure-preserving projection. Then,  there exists a symplectic matrix $A\in Sp({2k},\mathbb{R}^{2n})$ with $A_{qp}=0$. In the block form, this implies that
\begin{equation}\label{newA}
 A=
 \begin{bmatrix}
  A_{qq} & 0 \\
 A_{pq} & A_{pp} \\
 \end{bmatrix},
\end{equation}
$A_{qq}^TA_{pq}$ is symmetric, and  $A_{qq}^T A_{pp}=I_k$. Using (\ref{newA}), the projection operator $\sigma \circ \pi$ has the form
\begin{equation}\label{AA+}
 AA^+=
 \begin{bmatrix}
  A_{qq}A_{pp}^T & 0 \\
 A_{pq}A_{pp}^T-A_{pp}A_{pq}^T & A_{pp}A_{qq}^T \\
 \end{bmatrix},
\end{equation}
which gives an invariant subspace, $0 \oplus Q^*$, of $\mathbb{V}$. Thus, all vectors $u \in 0 \oplus Q^*$ are transformed by $\sigma \circ \pi$ into vectors that are also contained in $0 \oplus Q^*$. This can be stated as
\begin{equation}
u \in 0 \oplus Q^* \Rightarrow (\sigma \circ \pi) (u) \in 0 \oplus Q^*.
\end{equation}}
\end{remark}

\subsection{Proper symplectic decomposition (PSD)}\label{sec:psddiss}
 PSD is an  empirical model reduction technique, where the empirical data is used to construct a symplectic basis matrix $A$. Let $q(t_i), p(t_i)\in \mathbb{R}^{n}$  ($i=1,\ldots, N$) denote the empirical data.  Assume $n\ge 2N$. Rewriting the state variable in the form $x(t_i)=[q(t_i);p(t_i)]$, we can define a snapshot matrix in $\mathbb{R}^{2n\times N}$,
\begin{equation}\label{snap1}
M_x:=[x(t_1),\ldots,x(t_N)].
 \end{equation}
 The structure-preserving projection of $M_x$ onto a low dimensional subspace is given by $M_z=A^+M_x$, where $A\in Sp({2k},\mathbb{R}^{2n})$, $M_z=[z(t_1),\ldots, z(t_N)]\in \mathbb{R}^{2k\times N}$, and  $z(t_i)= A^+ x(t_i)$. The same projection of $M_x$ in the original coordinates is given by
$AM_z$, or $AA^+M_x$.

The Frobenius norm $\| \cdot \|_F$ can be used to measure the error between $M_x$ and its projection $A M_z$. Suppose a symplectic matrix $A$   minimizes the projection error in a least squares sense. Then, $A$ is a solution of the following  optimization problem:
\begin{equation}\label{opt1}
\begin{aligned}
&{\rm{minimize}} \quad  \|M_x-AA^+ M_x\|_F \\
 &{\rm subject \ to}  \quad  A^TJ_{2n}A=J_{2k} \quad {\rm{and}} \quad A_{qp}=0.
\end{aligned}
\end{equation}

Let $M_q=[q(t_1), \ldots, q(t_N)]$ and $M_p=[p(t_1), \ldots, p(t_N)]$. Using (\ref{AA+}), the cost function in (\ref{opt1}) can be expanded as
\begin{equation*}
\|M_q-A_{qq}A_{pp}^TM_q\|_F+\|M_p-(A_{pq}A_{pp}^T-A_{pp}A_{pq}^T) M_q -A_{pp}A_{qq}^T M_p\|_F.
\end{equation*}
By Remark \ref{rmk:3.8}, the constraint in (\ref{opt1}) holds if and only if $A_{qq}^TA_{pq}$ is symmetric and  $A_{qq}^T A_{pp}=I_k$.
Thus, (\ref{opt1}) is equivalent to
\begin{equation}\label{opt2}
\begin{aligned}
&{\rm{minimize}} \quad \|M_q-A_{qq}A_{pp}^TM_q\|_F+\|M_p-(A_{pq}A_{pp}^T-A_{pp}A_{pq}^T) M_q -A_{pp}A_{qq}^T M_p\|_F \\
 &{\rm subject \ to}  \quad  A_{qq}^TA_{pq}=A_{pq}^TA_{qq} \quad {\rm{and}} \quad A_{qq}^T A_{pp}=I_k.
\end{aligned}
\end{equation}

 Since matrices $A_{qq}$, $A_{pq}$ and $A_{qq}$ all have $n\times k$ elements and  the constraint region is nonconvex,  it is expected to be quite expensive to solve (\ref{opt2}) by nonconvex nonlinear programming.  To this effect, we use a singular value decomposition (SVD)-based method,  \emph{cotangent lift}, to construct a near optimal symplectic matrix  in a subset of $Sp(2k,\mathbb{R}^{2n})$ with $A_{qp}=0$.  The idea is to search for the optimal matrix, $A_1$, in a subset of $Sp(2k,\mathbb{R}^{2n})$ with $A_{qp}=0$, such that all the empirical data  lies near  ${\rm{Range}}(A_1)$. In particular, we assume that
 \begin{equation}\label{assump}
 A_{qp}=A_{pq}=0 \quad {\rm{and}} \quad A_{qq}=A_{qq}=\Phi.
 \end{equation}
 Then, $A_1 = {\rm{diag}}(\Phi, \Phi)$ for some $\Phi \in \mathbb{R}^{n\times k}$. It is straightforward to verify that $A_1^TJ_{2n}A_1=J_{2k}$ if and only if $\Phi^T\Phi=I_k$. Under these assumptions, (\ref{opt2}) reduces to
 \begin{equation}\label{opt3}
\begin{aligned}
&{\rm{minimize}} \quad \|M_q-\Phi\Phi^TM_q\|_F+\|P-\Phi \Phi^T P\|_F \\
 &{\rm subject \ to}  \quad  \Phi^T \Phi=I_k.
\end{aligned}
\end{equation}
The cost function in (\ref{opt3}) equals  $\|M_{q,p}-\Phi\Phi^T M_{q,p}\|_F$, where
\begin{equation}\label{exsnap}
M_{q,p}:=[q(t_1), \ldots,  q(t_N),  p(t_1), \ldots,  p(t_N) ]
\end{equation}
defines an extended snapshot matrix $M_{q,p}\in \mathbb{R}^{n\times 2N}$.

\begin{algorithm}
\caption{Cotangent Lift} \label{alg:lift}
\begin{algorithmic}
 \REQUIRE
An empirical data ensemble $\{q(t_i), p(t_i)\}_{i=1}^N$.
\ENSURE A symplectic matrix $A_1$ in block-diagonal form.
\STATE 1: Construct an extended snapshot matrix $M_{q,p}$ as (\ref{exsnap}).
\STATE 2: Compute the SVD of $M_{q,p}$ to obtain a POD basis matrix $\Phi$.
\STATE 3: Construct the symplectic matrix $A_1={\rm{diag}}(\Phi, \Phi)$.
\end{algorithmic}
\end{algorithm}

If $\Phi^*$ denotes the optimal value of $\Phi$ in (\ref{opt3}),  $\Phi^*$ can be directly solved by the SVD of $M_{q,p}$. Thus, the cotangent lift method simplifies the optimization problem (\ref{opt1}) to an SVD problem. Algorithm \ref{alg:lift} lists the detailed procedure of the cotangent lift method.   Since  the SVD of a $n\times m$ ($n\ge m$) matrix requires $2nm^2 + 2m^3$  operations \cite{TrefethenLN:97a}, the computational cost of Algorithm (\ref{alg:lift}) is $8nN^2+16N^3$, which is linearly dependent on $n$.

The cost function in (\ref{opt3}) is also equal to the projection error of the empirical data in the Frobenius norm  $\|M_x-AA^+ M_x\|_F$. Let  $\lambda_1\ge \lambda_2 \ge \ldots \ge \lambda_{2N}\ge 0$ denote the  singular values of $M_{q,p}$  in decreasing order. Then, the projection error of the cotangent lift method is  determined by the truncated singular values of $M_{q,p}$,
\begin{equation}\label{podtruc1}
E_{\rm{COT}}^{(2k)}={\left\| {(I - {\Phi^*}{\Phi^*}^T)M_{q,p}} \right\|_F} = \sqrt
{\sum\limits_{i = k+1}^{2N} {\lambda _i^2} },
\end{equation}
where we use the superscript $2k$ to emphasize that the symplectic subspace spanned by the column vectors of $A_1=({\rm{diag}}(\Phi, \Phi)$ has dimension $2k$.


\medskip
\begin{proposition} \label{2-approx}  Let $E_{\rm{COT}}^{(2k)}$, $E_{\rm{POD}}^{(2k)}$, and $E_{\rm{OPT}}^{(2k)}$  denote the  projection error  by the cotangent lift method, POD, and the nonlinear programming method to solve (\ref{opt1}), respectively. Then,
\begin{equation}
\frac{1}{2}E_{\rm{cot}}^{(4k)}\le E_{\rm{POD}}^{(2k)} \le E_{\rm{OPT}}^{(2k)} \le E_{\rm{COT}}^{(2k)}.
\end{equation}
\end{proposition}
\vspace{-4mm}
\begin{proof}
The cotangent lift method yields an optimal symplectic matrix in  $\mathbb{M}_1:=\{A\in \mathbb{R}^{2n\times 2k}: A={\rm{diag}}(\Phi, \Phi), \ \Phi=\mathbb{R}^{n\times k}, \ {\rm{and}} \ \Phi^T\Phi=I_k\}$. The feasible set of the optimization problem (\ref{opt1}) is given by  $\mathbb{M}:=\{A\in \mathbb{R}^{2n\times 2k}: AJ_{2n}A=J_{2k} \  {\rm{and}} \  A_{qp}=0\}$.  POD can find the most optimal matrix in $\mathbb{R}^{2n \times 2k}$ to minimize the projection error.  Thus, corresponding to $\mathbb{M}_1 \subset  \mathbb{M} \subset \mathbb{R}^{2n\times 2k}$, we have $E_{\rm{POD}}^{(2k)} \le E_{\rm{OPT}}^{(2k)} \le E_{\rm{COT}}^{(2k)}$.

Next we will prove  $\frac{1}{2}E_{\rm{cot}}^{(4k)} \le E_{\rm{POD}}^{(2k)}$.  According to  (\ref{podtruc1}), $(E_{\rm{cot}}^{(4k)})^2=\sum_{i = 2k+1}^{2N} {\lambda _i^2}$. Similar, if $\sigma_1\ge \sigma_2\ge \ldots \ge \sigma_N\ge0$ denotes all the singular values of $M_x$ in descending order, then $(E_{\rm{POD}}^{2k})^2={\sum_{i=k+1}^N \sigma_i^2}$.

Since $M_{q,p}=\begin{bmatrix}
  Q \\
 P \\
\end{bmatrix}$, $M_{q,p}^T M_{q,p}=
\begin{bmatrix}
  Q^TQ & Q^TP \\
 P^TQ & P^T P \\
\end{bmatrix}\in \mathbb{R}^{2N\times 2N}$, with eigenvalues $\{\lambda_i^2\}_{i=1}^{2N}$.
 Since $M_x=\begin{bmatrix}
  Q & P
\end{bmatrix}$, $M_x^T M_x=Q^TQ+P^TP \in \mathbb{R}^{N\times N}$, with eigenvalues $\{\sigma_i^2\}_{i=1}^N$. By the construction, both $M_{q,p}^T M_{q,p}$ and $M_x^T M_x$ are positive-semidefinite.  Let $S={\rm{diag}} \{M_x^TM_x, M_x^TM_x \}$, then $S$ is also positive-semidefinite, and the $i$th largest eigenvalue of $S$ is $(\sigma_{\lceil {{\frac{i}{2}}} \rceil})^2$.
 Moreover,
\begin{equation*}
\begin{aligned}
2S-M_{q,p}^T M_{q,p}&=
\begin{bmatrix}
  2M_x^TM_x & 0 \\
 0 & 2M_x^TM_x \\
\end{bmatrix}-
\begin{bmatrix}
  Q^TQ & Q^TP \\
 P^TQ & P^T P \\
\end{bmatrix} \\
&=
\begin{bmatrix}
  Q^TQ & -Q^TP \\
 -P^TQ & P^TP \\
\end{bmatrix}+
2\begin{bmatrix}
  P^TP & 0 \\
 0 & P^TP \\
\end{bmatrix}\\
&=
\begin{bmatrix}
  Q^T \\ -P^T
\end{bmatrix}
\begin{bmatrix}
  Q & -P
\end{bmatrix}+
2\begin{bmatrix}
  P^TP & 0 \\
 0 & P^TP \\
\end{bmatrix}
.
\end{aligned}
\end{equation*}
The last equation implies that $2S\ge M_{q,p}^T M_{q,p}\ge 0$. By the min-max theorem, the $i$th largest eigenvalue of $2S$ is greater than the $i$th largest eigenvalue of $M_{q,p}^T M_{q,p}$. This implies that $2(\sigma_{\lceil {{\frac{i}{2}}} \rceil})^2\ge \lambda_i^2$.
It follows that
\begin{equation*}
(E_{\rm{cot}}^{(4k)})^2= \sum_{i=2k+1}^{2N} \lambda_i^2 \le 2 \sum_{i=2k+1}^{2N}(\sigma_{\lceil {{\frac{i}{2}}} \rceil})^2=4 \sum_{i=k+1}^N \sigma_i^2=4 (E_{\rm{POD}}^{(2k)})^2.
\end{equation*}
This completes the proof.
\hfill
\end{proof}
\medskip


\subsection{Energy preservation}\label{sec:energypre}
Let $x(t)$ be the solution of (\ref{fom0}) with $x(0)=x_0$, and $E(t)= H(x(t))$ be the corresponding system energy at time $t$. Since $dH(x) \cdot X_H(x)=\Omega(X_H(x), X_H(x))=0$, the time derivative of $E(t)$ equals
\begin{equation}\label{engvar}
\dot E(t)= dH \cdot X |_{x(t)}= dH \cdot X_F |_{x(t)}.
\end{equation}
Thus, the time derivative of system energy is completely determined by the Hamiltonian function $H(x)$ and the vertical vector field $X_F(x)$.

\medskip
\begin{proposition} \label{prop:7.8.8}
The time derivative of $E(t)$  can also be represented in terms of the force field, i.e.
\begin{equation}\label{dissipativity}
 dH \cdot X_F|_x =\langle f_H,  \dot q \rangle|_x,
\end{equation}
where $\dot q (x)= \pi_Q  (X_H(x))=\nabla_p  H(q,p)$.
\end{proposition}
\begin{proof}
Let $X_F$ be a vertical vector field. By Proposition \ref{7.8.1}, $X_F$ induces a horizontal one-form $\Delta^F=-i_{X_F} \Omega$ on $\mathbb{V}$;  by Proposition \ref{prop:7.8.2}, $\Delta^F$ in turn induces a force field $f_H$, which is given by
\begin{equation}\label{fH}
\langle {f_H, \pi_Q(u)} \rangle  = {\Delta^F} \cdot u =  - \Omega (X_F, u),
\end{equation}
where $u$ is a vector field on  $\mathbb{V}$. Let $X_H$ denote the Hamiltonian vector field. Then,
\begin{align*}
 dH \cdot X_F &= (i_{X_H}\Omega )\cdot X_F = \Omega (X_H,X_F) =  - \Omega (X_F, X_H) = \langle f_H, \pi_Q  (X_H)) \rangle,
\end{align*}
which gives (\ref{dissipativity}).
\hfill
\end{proof}
\medskip

Let  $\tilde E(t)= \tilde H(z(t))$ denote the system energy of the forced Hamiltonian system (\ref{symprodiff}) in reduced coordinates. Similar to (\ref{engvar}), the time derivative of $\tilde E(t)$ is given by
\begin{equation}\label{reddiss}
\dot {\tilde E}(t)=d \tilde H \cdot X_{\tilde F}|_{z(t)}.
\end{equation}

\begin{theorem}\label{thm:preserving}
The reduced forced Hamiltonian system exactly preserves the time derivative of system energy at $z\in\mathbb{W}$, i.e. $d \tilde H \cdot X_{\tilde F}|_z=d  H \cdot X_{F} |_{\sigma(z)}$, if any one of the following conditions is satisfied at $\sigma(z)$:
\begin{enumerate}[label={\rm{(}}\alph*{\rm{)}}]
\setlength{\itemsep}{5pt}
\item $(\sigma \circ \pi)(X_F)=X_F$.
  \item $(\sigma \circ \pi)(X_H)=X_H$.
  \item $(\sigma \circ \pi)(X_H)=X$.
  \item $(\pi_P\circ \sigma \circ \pi)(X_F)=\pi_P(X_F)$.
  \item $(\pi_Q \circ \sigma \circ \pi)(X_H)=\pi_Q (X_H)$.
\end{enumerate}
\end{theorem}

\begin{proof}
Using (\ref{constantiden}), we obtain
\begin{equation}\label{idea1}
\begin{aligned}
d \tilde H \cdot X_{\tilde F}|_z&=\omega(X_{\tilde H}, X_{\tilde F})|_z = \omega(\pi (X_H(\sigma (z))), \pi (X_F(\sigma (z))))\\
&=\Omega(X_{H}, (\sigma \circ \pi) X_{F})|_{\sigma(z)}\\
&=\Omega((\sigma \circ \pi)X_{H}, X_{F})|_{\sigma(z)}.
\end{aligned}
\end{equation}
If $(a)$ or $(b)$ holds, the second or third line of (\ref{idea1}) would imply that $d \tilde H \cdot X_{\tilde F}|_z =\Omega(X_{H}, X_{F})|_{\sigma(z)}$, which equals $d H \cdot X_F|_{\sigma(z)}$.

Using (\ref{zerocond}), the time derivative of system energy for the reduced system can  be represented by
\begin{equation}\label{idea2}
d \tilde H \cdot X_{\tilde F}|_z=\Omega(X_{H}, (\sigma \circ \pi)(X_{H}+X_{F}))|_{\sigma(z)}=\Omega(X_{H}, (\sigma \circ \pi)(X))|_{\sigma(z)}.
\end{equation}
This implies that $(c)$ is  a sufficient condition to preserves the time derivative of system energy.

Since both $X_F$ and $(\sigma \circ \pi)X_F$ are vertical vector fields, we have $\pi_Q ((\sigma \circ \pi)X_F)=\pi_Q(X_F)=0$. Thus, $(a)$ is equivalent to $(d)$.

Finally, by (\ref{idea1}) and (\ref{fH}), we obtain
\begin{align*}
 d \tilde H \cdot X_{\tilde F}|_z &= \Omega((\sigma \circ \pi)(X_{H}), X_{F})|_{\sigma(z)}=-\Omega(X_{F}, (\sigma \circ \pi)(X_{H}))|_{\sigma(z)} \\
 &= \langle {f_H, \pi_Q((\sigma \circ \pi)(X_{H}))} \rangle|_{\sigma(z)}.
\end{align*}
Thus, if $(e)$ is satisfied, the time derivative of system energy is preserved as well.
\hfill
\end{proof}
\medskip

\subsubsection{Optimization of the basis matrix}
Theorem \ref{thm:preserving} implies that the time derivative of system energy is exactly preserved when the vector fields $X_H$, $X_F$,  $X$, or their vertical/horizontal components  are invariant under the projection operator $\sigma \circ \pi$. Motivated by this,  we formulate five optimization problems corresponding to each individual condition in Theorem \ref{thm:preserving}. All the optimization problems seek to construct a symplectic basis matrix $A$  such that one of the aforementioned vector fields (or their vertical/horizontal components) can lie near the subspace spanned by the column vectors of $A$.

\emph{Condition $(a)$.} The condition $(a)$ can be written as $AA^+X_F=X_F$ in canonical coordinates, which requires that $X_F(Az)\in {\rm{Range}}(A)$ for each $z$. To  satisfy this condition approximately, we can  construct an extended data ensemble,
\begin{equation}\label{snap2}
M_{x,X_F}:=[x(t_1),\ldots,x(t_N), X_F(x(t_1)),\ldots,X_F(x(t_N))],
 \end{equation}
and then  construct a symplectic matrix $A$ to fit each column vector of $M_{x,X_F}$ by solving the following optimization problem:
\begin{equation}\label{opt4}
\begin{aligned}
&{\rm{minimize}} \quad  \|M_{x,X_F}-AA^+ M_{x,X_F}\|_F \\
 &{\rm subject \ to}  \quad  A^TJ_{2n}A=J_{2k} \quad {\rm{and}} \quad A_{qp}=0.
\end{aligned}
\end{equation}

\emph{Condition $(b)$}. Replacing $X_F$ (resp., $M_{x,X_F}$) with $X_H$ (resp., $M_{x,X_H}$) in (\ref{snap2}) and solving (\ref{opt4}) to minimize $M_{x,X_H}$ will yield a symplectic matrix $A$ to  satisfy  $(b)$ approximately.

\emph{Condition $(c)$}. If we replace $X_F$  with $X$ in (\ref{snap2}), then $(c)$ will be approximately satisfied by the similar procedure. In this case, the symplectic matrix $A$ is constructed to fit both the solution snapshots $x$ and time derivative $X$ of $x$ simultaneously to preserve the time derivative of system energy. In previous literature, the analogous idea of incorporating time derivative snapshots \cite{FarhatC:11b,Mohseni:14a} or difference quotients~\cite{KunischK:01a, KunischK:02a, IliescuT:04a}, into the data ensemble has been widely used  to enhance the performance (such as convergence and accuracy) of POD reduced models.

\emph{Condition $(d)$}. In order to approximately satisfy $(d)$, we first construct a data ensemble in $\mathbb{R}^{n\times N}$ for the force field
\begin{equation}
M_{f_H}:=[f_H(x(t_1)),\ldots, f_H(x(t_N))].
 \end{equation}
 Since the vertical vector field $X_F(x)$ can be represented by $X_F(x)=[0; f_H(x)]$, the corresponding data ensemble for the vertical vector field is given by
 \begin{equation}
M_{X_F}:=\begin{bmatrix}
  0 \\
 M_{f_H} \\
\end{bmatrix},
 \end{equation}
 where $M_{X_F} \in \mathbb{R}^{2n\times N}$.  Thus, an optimal value of $A$ can be obtained by minimizing a cost function that is related to the projection error of $x$ and $\pi_P(X_F(x))$. In particular, the cost function can be formulated as
 \begin{equation}\label{MxF}
  \|M_x-AA^+ M_x\|_F+ \gamma \|I_{P} M_{X_F}-I_{P} AA^+ M_{X_F}\|_F,
 \end{equation}
 where $I_p=[0, I_n]\in \mathbb{R}^{n\times 2n}$ is the matrix representation of $\pi_P$, and $\gamma$ is a weighting coefficient to balance the truncation of $M_x$ and $I_{P} M_{X_F}$. Replacing $AA^+$ by (\ref{AA+}) simplifies the cost function (\ref{MxF}) to
 \begin{equation}\label{MxF2}
  \|M_x-AA^+ M_x\|_F+ \gamma \|M_{f_H}-A_{pp}A_{qq}^T M_{f_H}\|_F.
 \end{equation}
 When $\gamma=1$,  the cost functions in (\ref{opt4}) and (\ref{MxF2}) are exactly the same.

\emph{Condition $(e)$}.
In order to satisfy $(e)$ approximately, one can construct a data ensemble that contains $\nabla_p H(x)$, and form an optimization function in terms of $I_Q=[I_n, 0]\in \mathbb{R}^{n\times 2n}$.

While nonconvex nonlinear programming can result in the most optimal subspace to fit an extended data ensemble, considering $A$ is a matrix with $2n\times 2k$ elements, the programming problem can be very expensive and even intractable. Thus, we shall propose a cotangent lift method to obtain a near optimal value of $A$ at a relatively lower cost while simultaneously preserving the time derivative of system energy.

 \subsubsection{Cotangent lift with energy preservation}
The cotangent lift method can simplify all the optimization problems mentioned in the previous section. As an example, we shall give a cotangent lift algorithm to  minimize the cost function (\ref{MxF2}) with $\gamma=1$, such that $(a)$ and $(d)$ can be satisfied.

 The cotangent lift methods requires that $A_{qp}=A_{pq}=\Phi$ and $A_{qq}=A_{pp}=\Phi$, where $\Phi$ is an orthonormal matrix.  Then, $q$, $p$, and $f_H$ have the same status in the cost function, and all the data of $q$, $p$, and $f_H$  should lie near the Range of $A$. As a result, the $k$ columns of $\Phi$ can be obtained from the left singular vectors of the following data ensemble
 \begin{equation}\label{pqfH}
M_{q,p,f_H}:=[q(t_1), \ldots,  q(t_N),  p(t_1), \ldots,  p(t_N), f_H(x(t_1)), \ldots, f_H(x(t_N))].
 \end{equation}

 \begin{algorithm}
\caption{Cotangent lift with energy preservation} \label{alg:lift1}
\begin{algorithmic}
 \REQUIRE
An empirical data ensemble $\{q(t_i), p(t_i), f_H(x(t_i))\}_{i=1}^N$.
\ENSURE A symplectic matrix $A_1$ in block-diagonal form.
\STATE 1: Construct an extended snapshot matrix $M_{q,p,f_H}$ as (\ref{pqfH}).
\STATE 2: Compute the SVD of $M_{q,p,f_H}$ to obtain a POD basis matrix $\Phi$.
\STATE 3: Construct the symplectic matrix $A_1={\rm{diag}}(\Phi, \Phi)$.
\end{algorithmic}
\end{algorithm}

 Algorithm \ref{alg:lift1} lists the detailed procedure for the cotangent lift method for the preservation of the time derivative of system energy. Since this algorithm is based on SVD, the reconstruction error of the time derivative of system energy at $x=\sigma(z)$ can be estimated by the following:
\begin{equation*}
\begin{aligned}
\|d  H \cdot X_{F}|_x-  d \tilde H \cdot X_{\tilde F}|_z\|&=\|\Omega(X_{H}, X_{F})|_x-\Omega(X_{H}, (\sigma \circ \pi)(X_{F}))|_x\| \\
&\le  \|X_H(x)\| \cdot \|X_F(x)-AA^+X_F(x)\|\\
&= \|X_H(x)\| \cdot  \|f_H(x)-A_{pp}A_{qq}^T f_H(x)\|.
\end{aligned}
 \end{equation*}
In a compact subset $M$ of $\mathbb{V}$, we can assume $\|X_H(x)\|$ to be uniformly bounded. If the data set of $f_H(x)$ is representative at the solution trajectory, $\|f_H(x)-A_{pp}A_{qq}^T f_H(x)\|$ is bounded by a constant multiplied by the truncated singular values of (\ref{pqfH}).

\section{Reduction of dissipative Hamiltonian systems}\label{sec:dissipative}
In this section, we discuss a special form of  forced Hamiltonian systems where the vertical vector field is dissipative. We also prove that the proposed model reduction method  preserves the stability of the dissipative Hamiltonian system.

\subsection{Dissipative Hamiltonian systems}
We begin with the definition of a dissipative vector field.

\medskip
\begin{definition}
A vertical vector field $X_F$ on $\mathbb{V}$ is called  dissipative if $dH \cdot X_F|_x \le 0$ for every $x\in\mathbb{V}$.
\end{definition}
\medskip

Proposition \ref{prop:7.8.8} implies that a vertical vector field $X_F$ is dissipative if and only if the corresponding force field $f_H$  satisfies $\langle f_H, \dot q \rangle|_{(q,p)} \le 0$ at all  $(q,p)\in \mathbb{V}$.
\medskip

\begin{definition}
A forced Hamiltonian system (\ref{fom0}) is  dissipative if the vector field can be decomposed as $X=X_H+X_F$, where $X_H$ is a Hamiltonian vector field and $X_F$ is a dissipative vector field.
\end{definition}
\medskip

By (\ref{engvar}), if the vertical vector field $X_F$ is  dissipative, then  $\dot E(t)\le 0$, which means that the system energy is nonincreasing in time.

In the last section, based on the empirical data of $X_F$, $X_H$, $X$, or $f_H$, we have discussed several approaches to extend the snapshot matrix  such that the reduced model can quantitatively preserve the rate of energy dissipation.  In the absence of the empirical data of vector fields, it is still desired for the reduced model to qualitatively preserve the \emph{dissipativity}. This implies that if the original system is dissipative, then the reduced system should remain dissipative.   Fortunately, when the dissipation is Rayleigh dissipation, the aforementioned structure-preserving projection automatically preserves the dissipativity, and this property is independent from the data that is used to construct the basis matrix $A$.


The dissipative force often arises from Rayleigh dissipation function, which can be written as
\begin{equation}
\mathcal{F}(q,\dot q)=\frac{1}{2} \dot q^T R(q) \dot q
\end{equation}
in Lagrangian coordinates, where $R(q)\in \mathbb{R}^{n\times n}$ is a symmetric positive-semidefinite matrix. The force field is then given by $f_L(q,\dot q)=-\nabla_{\dot q} \mathcal{F}(q,\dot q) =- R(q) \dot q$. Using  the Legendre transformation, we obtain $f_H(q,p)=- R(q) \dot q(q,p)$ in Hamiltonian coordinates. Since
$$\langle f_H, \dot q\rangle|_{(q,p)} = -\dot q^T R \dot q|_{(q,p)}  \le 0,$$
this verifies that the corresponding vertical vector field $X_F$ is dissipative.

If the reduced system is constructed by the structure-preserving reduction, then $A_{qp}=0$ and $q=A_{qq}r+A_{qp}s=A_{qq}r$. If follows that  $\dot q=A_{qq} \dot r$. Using (\ref{reduceddiff}), the reduced force field is given by $A^T_{qq} \tilde f_H(r,s)= -A_{qq}^T R(A_{qq} r) A_{qq} \dot r$. Thus, the rate of energy variation of the reduced system at $(r,s)$ is given by
$$\langle  A_{qq}^T \tilde f_H, \dot r\rangle|_{(r,s)} =- \langle  A_{qq}^T R(A_{qq} r) A_{qq} \dot r, \dot r \rangle|_{(r,s)} = -( A_{qq}\dot r)^T R(A_{qq} r) ( A_{qq}\dot r)|_{(q,p)}  \le 0.$$
 This verifies that the reduced system preserves the dissipativity.

Dissipativity preservation often is a strong indicator for stability preservation, as discussed in the next section.

\subsection{Stability preservation}
Let $\mathbb{V}=\mathbb{R}^{2n}$ be a configuration space with the standard topology induced by the Euclidean norm $\|\cdot \|$.  Let $M$ be a subset of $\mathbb{V}$. Then the subspace topology in $M$ is the same as the metric topology obtained by restricting the Euclidean norm $\|\cdot \|$ to $M$. Since the Hamiltonian function $H: \mathbb{V}\to \mathbb{R}$ is continuous, the restriction of $H$ to $M$ gives a continuous function $H_M:M\to \mathbb{R}$.  Throughout this section, we assume that the forced Hamiltonian system is  dissipative, and  the solution $x(t)$ of the system lie in $M$ for every $t\ge 0$. We use $E_0:=H(x_0)$ to denote the system energy at $t=0$.

Let $x(\mathbb{R_+})=\{x(t): t\ge 0\}$ denote the solution trajectory of a dissipative Hamiltonian system. We say the system is \emph{uniformly bounded} if  there exists a closed $r$-ball $B_r:=\{x\in M: \|x\| \le r\}$ centered at $0$ such that $x(\mathbb{R_+}) \subset B_r$.
Under certain conditions, the dissipative Hamiltonian system is uniformly bounded, as the following three lemmas indicate.

\medskip
 \begin{lemma}\label{lem:D}
 Let $D:=H_M^{-1}((-\infty, E_0])=\{x\in M: H_M(x)\le E_0\}$ denote a sublevel set of the Hamiltonian function $H_M:M \to \mathbb{R}$. Let $D_0$ be the connected component of $D$ that contains $x_0$. If $D_0$ is bounded, then  the  dissipative Hamiltonian system is uniformly bounded.
 \end{lemma}
\begin{proof}
Since $H\circ x: t  \mapsto H_M(x(t))$ gives a continuous function of $t$, the set $x(\mathbb{R_+})=\{x(t): t\ge 0\}$ is path connected, hence it is connected. Since the forced Hamiltonian system is  dissipative, $H_M(x(t))\le E_0$ for any $t\ge 0$. This implies that $x(\mathbb{R_+}) \subset D$. Since $D_0$ is a connected component of $D$ and $D_0\cap x(\mathbb{R_+})$ contains $x_0$,  the connected set $x(\mathbb{R_+})$  lies entirely within $D_0$. Hence, if $D_0$ is bounded, so is $x(\mathbb{R_+})$.
 \hfill
 \end{proof}
 \medskip

\begin{lemma} \label{lem:U}
  If there exists a bounded neighborhood $U$ of $x_0$ in $M$ such that $E_0 < H_M(x)$  for every  $x$ on the boundary of $U$, then the  dissipative Hamiltonian system  is uniformly bounded.
\end{lemma}
\begin{proof}
Let ${\rm{bd}}_M (U)$ denote the boundary of $U$ in $M$, and ${\rm{cl}}_M(U)$ denote the closure of $U$ in $M$. Since $H_M(x)> E_0$  for every  $x\in {\rm{bd}}_M (U)$, we have $D\subset M-{\rm{bd}}_M (U)$. Since $U$ and $M-{\rm{cl}}_M(U)$ form a separation of $M-{\rm{bd}}_M (U)$, as a connected set, $D_0$ must lie entirely within either   $U$ or $M-{\rm{cl}}_M (U)$. Since $x_0\in D_0\cap U$, the only possible case is that $D_0 \subset U$. Because $U$ is bounded, so is $D_0$. By Lemma \ref{lem:D}, the  dissipative Hamiltonian system is uniformly bounded.
\hfill
\end{proof}

\medskip
\begin{lemma} \label{cor:S}
If $\mathop {\lim}\limits_{x \to \infty}
  H_M(x)=+\infty$ in $M$, then  the  dissipative Hamiltonian system is uniformly bounded.
\end{lemma}
\begin{proof}
Suppose the    system is not uniformly bounded. Then there exists an increasing  sequence of time $\{t_1, t_2, \ldots\}$ such that $ \|x(t_i)\|>i$ for each $i\in \mathbb{N}_+$. By assumption, $H_M(x) \to +\infty$ as $x\to \infty$. Thus, for any $E_0\in \mathbb{R}$, there exists an $n\in\mathbb{N}_+$ such that as long as $\|x\|>n$, $H_M(x)>E_0$. This implies that $H_M(x(t_i))>E_0$ for every $i\ge n$. But if the system is dissipative, we must have $H_M(x(t_i)) \le E_0$, which is a contradiction.
\hfill
\end{proof}
\medskip

\begin{remark}
\emph{If $E_0$ is a regular value of $H_M:M\to \mathbb{R}$, then the level set $H_M^{-1}(E_0)$ is an embedded codimension-1 submanifold in $M$ by the regular value theorem, and the sublevel set $D$ is an embedded codimension-0 submanifold with boundary in $M$~\cite{LeeJM:12a} (pp. 120--121).}
\end{remark}
\medskip

Let $M=\mathbb{V}$. Then, Lemmas \ref{lem:D}, \ref{lem:U}, and  \ref{cor:S} imply that under certain conditions, the original  dissipative Hamiltonian system is bounded. Moreover, if we respectively replace $x_0$ and $E_0$  by $x(t_1)$ and $H(x(t_1))$ for some $t_1\in \mathbb{R}$, these lemmas still hold.

Next, we consider boundedness  of the structure-preserving reduced model. Suppose that the reduced  system remains  dissipative, $x_0\in {\rm {Range}}(A)$, and the initial condition of the reduced system is given by $z_0=A^+x_0$. Let $M= {\rm {Range}}(A)$.
Then, Lemmas \ref{lem:D}, \ref{lem:U}, and  \ref{cor:S} imply that under the same conditions, the reduced  dissipative Hamiltonian system preserves the boundedness. In particular, in Lemma \ref{lem:D}, if the connected component $D_0$ of $H^{-1}((-\infty, E_0])$ in $\mathbb{V}$ is bounded, then the connected component of $D_0\cap M$ that contains $x_0$ is bounded in $M$.
In Lemma \ref{lem:U}, if there exists a bounded neighborhood $U$ of $x_0$ in $\mathbb{V}$ such that $E_0 < H(x)$  for every  $x\in {\rm{bd}}_\mathbb{V} (U)$, then $U_M:=U \cap M$ is a neighborhood of $x_0$ in $M$ and is bounded in $M$. Moreover, since ${\rm{bd}}_M (U_M)\subset {\rm{bd}}_\mathbb{V} (U)\cap M$, $E_0<H (x)$ for every $x\in {\rm{bd}}_M (U_M)$.
In Lemma \ref{cor:S}, if $\mathop {\lim}\limits_{x \to \infty}  H(x)=+\infty$ in $\mathbb{V}$, then $\mathop {\lim}\limits_{x \to \infty} H_M(x)=+\infty$ in $M$.

Under the assumptions of Lemmas \ref{lem:D}, \ref{lem:U}, and  \ref{cor:S},  we have proved that the boundedness of the original and the reduced  systems is consistent. In dynamical systems, boundedness is often accompanied with stability. An equilibrium point $x_*$ of a dynamical system  is \emph{Lyapunov stable} if for every neighbourhood $U$ of $x_*$, there exists a neighbourhood $V\subset U$ such that if $x_0 \in V$, then $x(t)\in U$ for every $t\ge 0$. When the sysetem is linear and  uniformly bounded, it is marginally stable in the sense of Lyapunov. If the original forced Hamiltonian system is linear, then the reduced system constructed by the structure-preserving projection is also linear. Thus, if any assumption in the previous lemmas holds,  both the original and reduced systems are Lyapunov stable.

 \medskip
 \begin{theorem}\label{minimal}
 Let $M$ be a closed subset of $\mathbb{V}$. If $x_*$ is a strict local minimum of $H_M$ in $M$, then $x_*$ is a stable equilibrium for  the dissipative Hamiltonian system.
 \end{theorem}
 \medskip

 \begin{proof}
  Since $x_*$ is a strict local minimum of $H_M$, then there exits a neighbourhood $W$ of $x_*$ such that $H_M(x)> H_M(x_*)$ for every $x\in {\rm{cl}}_M(W)-\{x_*\}$. Assume $W$ is bounded in $M$, otherwise,  replace $W$ by $W\cap B_r(x_*)$ for an open $r$-ball $B_r(x_*)$  centered at $x_*$. Let $U$ be an arbitrary neighborhood of $x_*$ in $M$.  Since both $W$ and $U$ are open, so is $W\cap U$. Let $U_0=W\cap U$.   Since ${\rm{bd}}_M (U_0) \subset {\rm{cl}}_M (W)$, ${\rm{bd}}_M (U_0)$ is bounded in $M$, and hence also bounded  in $\mathbb{V}$. Since $M$ is closed in $\mathbb{V}$,  ${\rm{bd}}_M(U_0)$ is also closed in $\mathbb{V}$. As a bounded and  closed subset of $\mathbb{V}$, ${\rm{bd}}_M(U_0)$ is compact. By the extreme value theorem, there exits $x_1 \in {\rm{bd}}_M(U_0)$ such that $H_M(x_1) \le H_M(x)$ for every $x\in {\rm{bd}}_M(U_0)$.
   Since $H_M$ is continuous, the preimage $D:=H_M^{-1}((-\infty, H_M(x_1)))$ of $(-\infty, H(x_1))$ is open  in $M$. Since $x_1\in {\rm{cl}}_M(W)$ and $x_1\ne x_*$, we have $H_M(x_*)<H_M(x_1)$, which implies that $x_*\in D$. Thus, $V:=U_0\cap D$ is  a neighbourhood of $x_*$ in $M$.  If $x_0\in V$, then $H_M(x_0)<H_M(x_1)\le H_M(x)$ for every $x\in {\rm{bd}}_M (U_0)$. This implies that $x(\mathbb{R}_+)\subset U_0\subset U$, by Lemma \ref{lem:U}. Therefore,  $x_*$ is a stable equilibrium for  the dissipative Hamiltonian system.
   \hfill
 \end{proof}
 \medskip

Let $M=\mathbb{V}$. Suppose $U$ is a neighborhood of $x_*$ in $\mathbb{V}$, and $x_*$   is the minimum  of $H$ in $U$. Then, Theorem \ref{minimal} implies that the full model is stable at $x_*$. Now, let $M={\rm{Range}}(A)$. It immediately follows that $x_*$   is also the minimum of $H_M$ in $U_M$, where $U_M= U \cap M$. It follows that $x_*$ is also the stable equilibrium of the reduced Hamiltonian system on ${\rm{Range}}(A)$. Therefore, the stability of the full and reduced dissipative Hamiltonian systems  is consistent. For both the full model and the structure-preserving reduced model,  $H_M$ can be considered a Lyapunov function for the system. Nevertheless, a POD reduced system is not guaranteed to be dissipative and stable, and therefore, there is no corresponding Lyapunov function.

While both the structure-preserving method and the POD-Galerkin method  construct reduced equations in some low dimensional subspaces, only the structure-preserving method  can preserve  the forced-Hamiltonian structure.   The PSD algorithm can be used to construct a symplectic matrix $A$, which is an analogous to POD that constructs an orthonormal basis matrix $\Phi$. Evolving a  PSD reduced system  by a symplectic integrator can capture the energy variation and preserve the stability. By contrast, even if a POD subspace can fit the empirical data with good accuracy, a POD reduced system can be unstable. To this end, one can distinguish between a numerically reduced system and a physically reduced system. Table \ref{tab:podvspsd} compares the POD-Galerkin method with the proposed structure-preserving  model reduction  method; it serves as a short summary of sections \ref{sec:twoapp}--\ref{sec:dissipative}.

\begin{table}
\begin{center}
\caption{The POD-Galerkin method vs. the structure-preserving  model reduction method.}
 \label{tab:podvspsd}
 \resizebox{\textwidth}{!}{%
\begin{tabular}{|c|c |c|c|}
\hline
&  POD-Galerkin   & \multicolumn{2}{c|}{Structure-preserving model reduction}\\
\hline
Original system &  \begin{tabular}{@{}c@{}} General ODE system: \\ $\dot x=f(x)$ with $x\in \mathbb{R}^n$  \end{tabular} &  \multicolumn{2}{c|}{ \begin{tabular}{@{}c@{}} Forced Hamiltonian system: \\ $\dot x= X_H(x)+ X_F(x)$ with  $x\in \mathbb{R}^{2n}$  \end{tabular}  } \\
\hline
\begin{tabular}{@{}c@{}} Physical laws of \\ the original system \end{tabular}
 &  Newton's Law   &  \begin{tabular}{@{}c@{}} Integral d'Alembert's \\ principle \end{tabular} & \begin{tabular}{@{}c@{}}  Local d'Alembert's \\ principle \end{tabular} \\
\hline
Reduced state & \begin{tabular}{@{}c@{}} Orthogonal projection: \\   $z=\Phi^T x\in \mathbb{R}^k$ \end{tabular}
& \multicolumn{2}{c|}{ \begin{tabular}{@{}c@{}}  Symplectic projection:   \\  $z=A^+ x\in \mathbb{R}^{2k}$ \end{tabular}}  \\
\hline
 Reduced system &  \begin{tabular}{@{}c@{}} Reduced ODE system: \\  $\dot z =\Phi^T f(\Phi z)$ \end{tabular} & \multicolumn{2}{c|}{\begin{tabular}{@{}c@{}} Reduced forced Hamiltonian system: \\ $\dot z = X_{\tilde H}(z)+X_{\tilde F}(z)$\end{tabular} }       \\
  \hline
Reduction approach & Galerkin projection   & Variational principle & Structure-preserving projection \\
  \hline
  \begin{tabular}{@{}c@{}} Physical laws of \\ the reduced system \end{tabular}
 &  N/A  &  \begin{tabular}{@{}c@{}} Integral d'Alembert's \\ principle \end{tabular} & \begin{tabular}{@{}c@{}}  Local d'Alembert's \\ principle \end{tabular} \\
 \hline
 Basis matrix & Orthonormal: $\Phi^T \Phi=I_k$ &  \multicolumn{2}{c|}{Symplectic: $A^TJ_{2n}A=J_{2k}$ and $A_{qp}=0$}\\
 \hline
 \begin{tabular}{@{}c@{}} Basis matrix \\ construction method  \end{tabular}
 & POD: SVD &  \multicolumn{2}{c|}{PSD: Cotangent lift } \\
 \hline
  Dissipativity &  N/A &   \multicolumn{2}{c|}{Dissipativity preservation} \\
 \hline
 Stability &  N/A &   \multicolumn{2}{c|}{Stability preservation} \\
 \hline
\end{tabular}}
\end{center}
\end{table}

\section{Numerical validation} \label{sec:numerical_diss}
 In this section, the performance of the proposed structure-preserving model reduction method is illustrated in numerical simulation of a linear dissipative wave equation. Our goal is  to demonstrate that PSD can deliver a low-dimensional reduced system while preserving the stability of the original system.

\subsection{Hamiltonian formulation of dissipative wave equations}
Let $u=u(t,x)$.  Consider a one-dimensional linear wave equation with constant damping coefficient $\beta$, undamped angular frequency $\omega_0$, and moving speed $c$,
\begin{equation}\label{lwave}
  u_{tt} +\beta u_t -  c^2   u_{xx}+ \omega_0^2 u=0,
\end{equation}
on space $x\in[0,l]$. With the generalized coordinates $q=u$ and the generalized momenta $p= u_t$, the Hamiltonian PDE associated with (\ref{lwave}) is given by
\begin{equation}\label{qp}
 \dot q= \frac  {\delta {H}}{\delta p}, \qquad
  \dot p= -\frac {\delta  {H}}{\delta q}-\beta p,
\end{equation}
where the Hamiltonian is defined as
  \begin{equation}\label{hamiqp}
 H(q,p) = \int_0^l {dx\left[ \frac{1}{2}p^2 + \frac{1}{2} \omega_0^2 q^2+ \frac{1}{2}c^2q_x^2 \right]}.
  \end{equation}

A fully resolved model of (\ref{qp}) can be constructed by a structure-preserving finite difference discretization~\cite{BridgesTJ:06a}. In particular, with $n$ equally spaced grid points, the spatially discretized Hamiltonian with periodic boundary conditions is given by
\begin{equation}\label{waveeng}
H_d(y) = \frac{\Delta x}{2} \sum\limits_{i = 1}^n p_i^2 + \frac{\omega_0^2 \Delta x}{2} \sum\limits_{i = 1}^n q_i^2 +  \frac{c^2}{2\Delta x}\sum\limits_{i = 1}^n (q_i - q_{i - 1})^2,
\end{equation}
where $x_i = i\Delta x$, $q_i= u(t,x_i)$, $q_0=q_n$, $p_i= u_t(t,x_i)$, and $y =[q_1; \ldots; q_n; p_1; \ldots; p_n]$. With $n\Delta x=l$, (\ref{waveeng}) converges to (\ref{hamiqp}) in the limit $\Delta x \to 0$. Now, we have a Hamiltonian ODE system,

\begin{equation}\label{wavedis}
\frac{{{\rm{d}}{y}}}{{{\rm{d}}t}} = {J_d}\nabla_y {H_d}+X_F,
\end{equation}
where ${J_d} =  J_{2n} / {\Delta x}$, and $X_F=[0; \ldots; 0; -\beta p_1; \ldots; -\beta p_n]$.
Let $D_{xx}\in \mathbb{R}^{n\times n}$ denote the three-point central difference approximation for the spatial derivative $\partial_{xx}$. We define a Hamiltonian matrix $K$ and a dissipative matrix $L$ by
 \begin{equation}\label{wavematrix}
 K=
 \begin{bmatrix}
 0_n & I_n \\  c^2 {D}_{xx}-\omega_0^2 I_n & 0_n
 \end{bmatrix},
 \quad
  L=
 \begin{bmatrix}
 0_n & 0_n \\  0_n & -\beta I_n
 \end{bmatrix}.
 \end{equation}
 Then, (\ref{wavedis}) can be written in the form
\begin{equation}\label{expansion}
\dot y=Ky+Ly.
\end{equation}
  Time discretization of (\ref{expansion}) can be achieved by using an implicit symplectic integrator scheme based on mid-point rule~\cite{HairerE:06a, McLachlanRI:06a}.

\subsection{Numerical results}
For our numerical experiments, we  study the one-dimensional dissipative wave equation with  periodic boundary conditions defined in (\ref{lwave}). Let  $s(x)=10\times |x-{\frac{1}{2}} |$; and let $h(s)$ be a cubic spline function:
\begin{equation*}
\begin{array}{l}
h(s)=\left\{
\begin{array}{ccl}
 \vspace{3pt} 1-\frac{3}{2}s^2+\frac{3}{4}s^3&  \text{if }&  0\le s \le 1 \\
 \frac{1}{4}(2-s)^3&  \text{if }&  1<s\le2 \\ 0&  \text{if }&  s>2
\end{array} \right. .
\end{array}
\end{equation*}
The initial condition is provided by
\begin{equation}\label{initial}
q(0)=[h(s(x_1));\ldots; h(s(x_n))], \quad p(0)=0_{n\times 1},
\end{equation}
which gives rise to a dissipative system with wave propagating in both directions of $x$ and then bouncing back.  The full model  is computed using the following parameter set:
\medskip
\begin{center}
\begin{tabular}{r||l}
  \hline
  Size of the space domain & $l=1$ \\
  Number of grid points & $n=500$ \\
  Space discretization step & $\Delta x=l/n=0.002$ \\
  Final time & $T=50$\\
  Time discretization step & $\delta t=0.01$ \\
  Damping coefficient & $\beta=0.1$ \\
  Undamped angular frequency & $\omega_0=0.05$\\
  Wave speed & $c=0.1$ \\
  \hline
\end{tabular}
\end{center}
\medskip
The reduced PSD model is constructed through the cotangent lift method based on the extended snapshot matrix (\ref{exsnap}) that contains snapshots of $q(t)$ and $p(t)$. Since $f_H=-\beta p$, this extended snapshot matrix can also discover the dominant modes of $f_H$, and therefore approximately preserves the system energy. Since (\ref{lwave}) is linear, we can also obtain the analytical solution by the eigenfunction expansion method. The analytical solution is used as the reference benchmark solver to measure the error of the full model as well as POD and PSD reduced models.

Figure \ref{fig:wavepod}(a) plots several snapshots of the solution profile from $t = 0$ to $t = 10$. The empirical data ensemble takes $101$ snapshots from the full model with uniform interval ($\Delta t=0.5$). We first compare PSD with the full model. The lines show the results from the full model  and the symbols show the results from the PSD reduced model with 20 modes. For all snapshots, the PSD reduced system obtains good results that match the full model very well.  Figure \ref{fig:wavepod}(b) shows the singular values  corresponding to the first 80 POD and PSD modes. A fast decay of singular values indicates a fast convergence of low-dimensional data to fit the original data with respectic to the $L^2$ norm.   Since POD  is designed to minimize the projection error of the data snapshots in least-squares sense, for a fixed dimension, no other linear projection method can provide better data approximation with the $L^2$ norm. With the symplectic constraint, we do observe that the cotangent lift requires more modes to fit the empirical data than POD in order to obtain the same accuracy.

\begin{figure}
\begin{center}
\begin{minipage}{0.48\linewidth} \begin{center}
\includegraphics[width=1\linewidth]{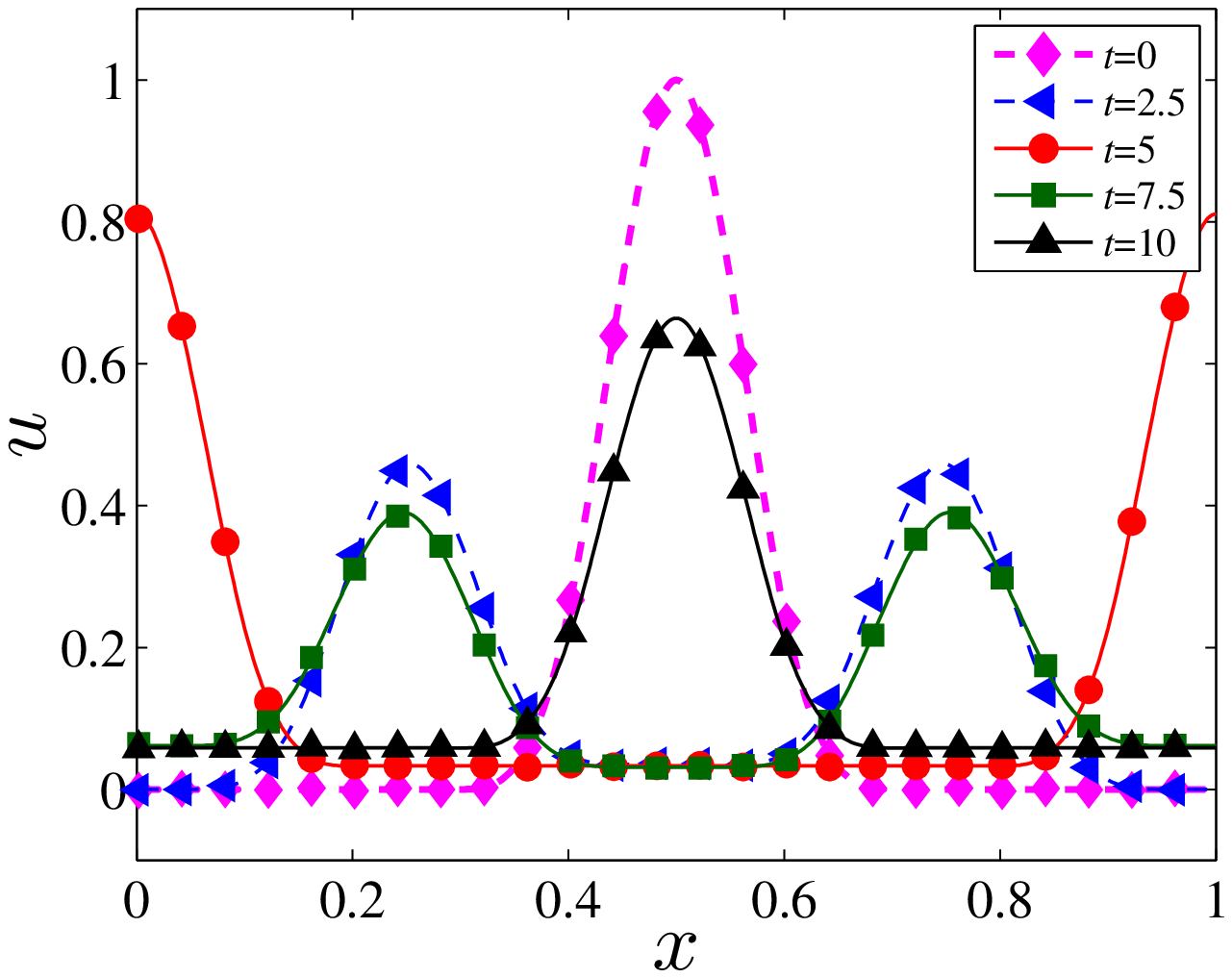}
\end{center} \end{minipage}
\begin{minipage}{0.48\linewidth} \begin{center}
\includegraphics[width=1\linewidth]{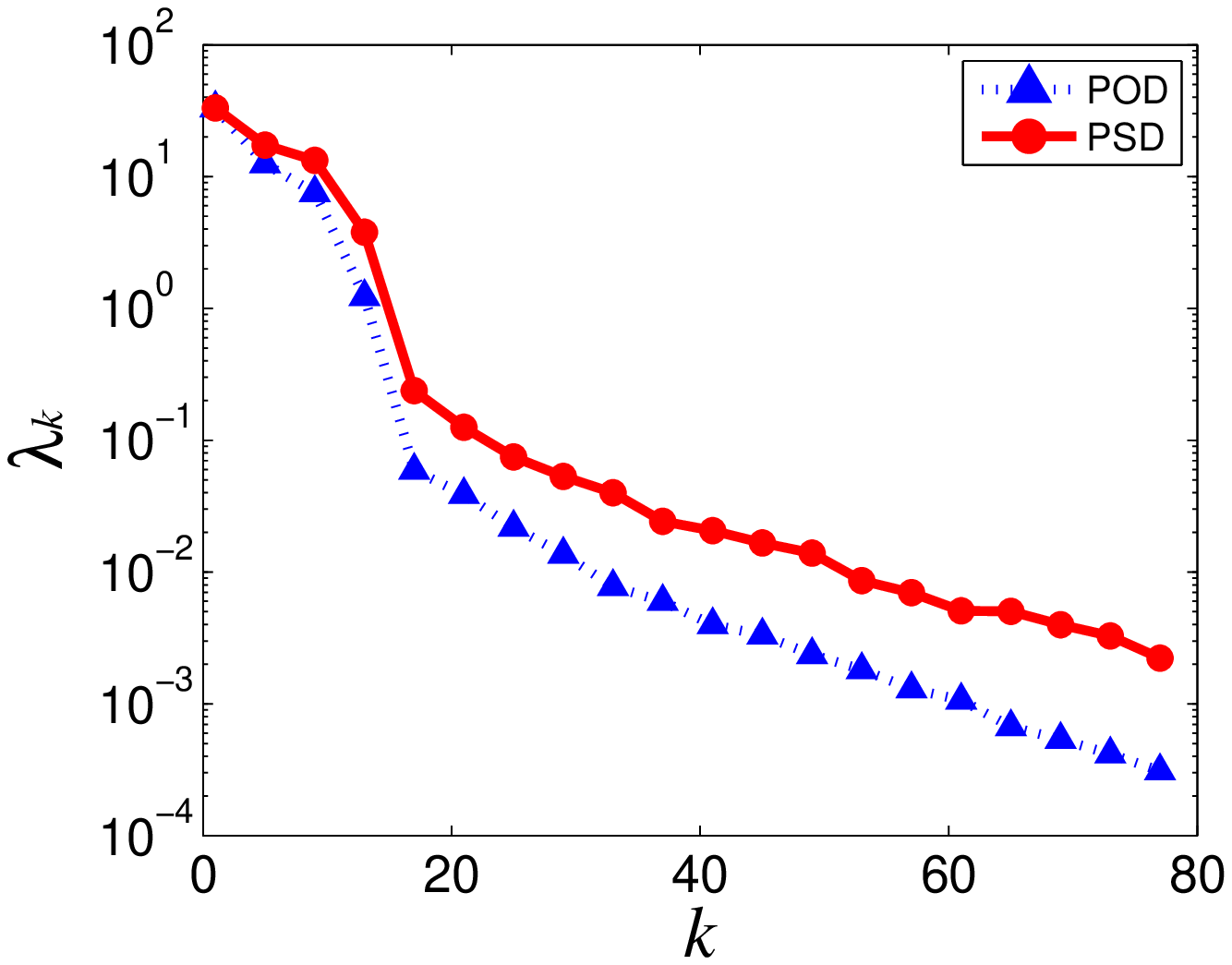}
\end{center} \end{minipage}\\
\begin{minipage}{0.48\linewidth}\begin{center} (a) \end{center}\end{minipage}
\begin{minipage}{0.48\linewidth}\begin{center} (b) \end{center}\end{minipage}
 \caption{(Color online.) (a) The solution $u(t,x)$ at $t=0, 2.5, 5, 7.5, 10$ of the dissipative wave equation. The lines represent the results from the full model based on 500 grid points and the symbols represent the results from the PSD reduced model with 20 modes.  (b) The singular values $\lambda_k$ corresponding to the first $k=1, 2, \ldots, 80$ POD and PSD modes.
} \label{fig:wavepod}
 \vspace{-3mm}
\end{center}
\end{figure}

\begin{figure}
\begin{center}
\begin{minipage}{0.48\linewidth} \begin{center}
\includegraphics[width=1\linewidth]{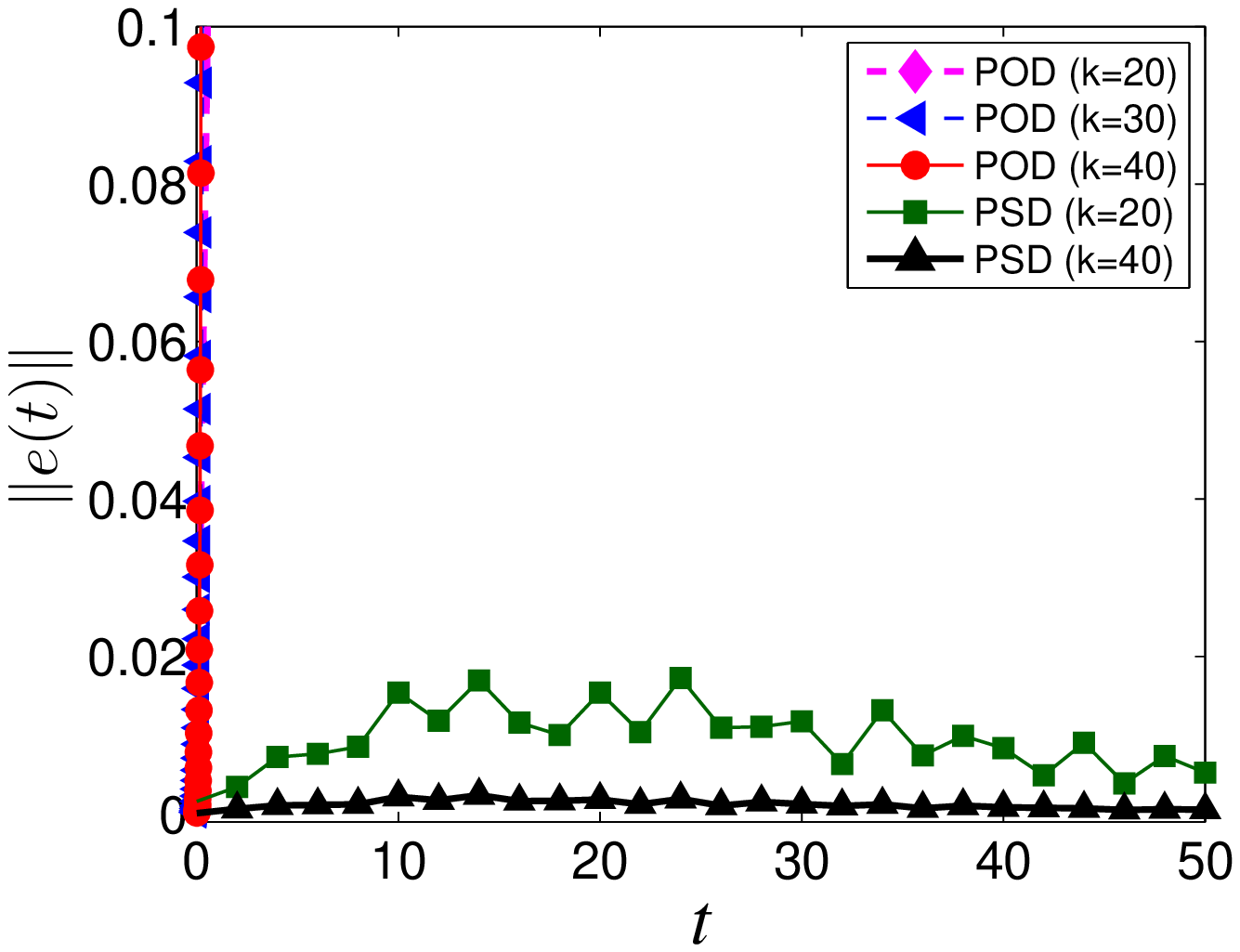}
\end{center} \end{minipage}
\begin{minipage}{0.48\linewidth} \begin{center}
\includegraphics[width=1\linewidth]{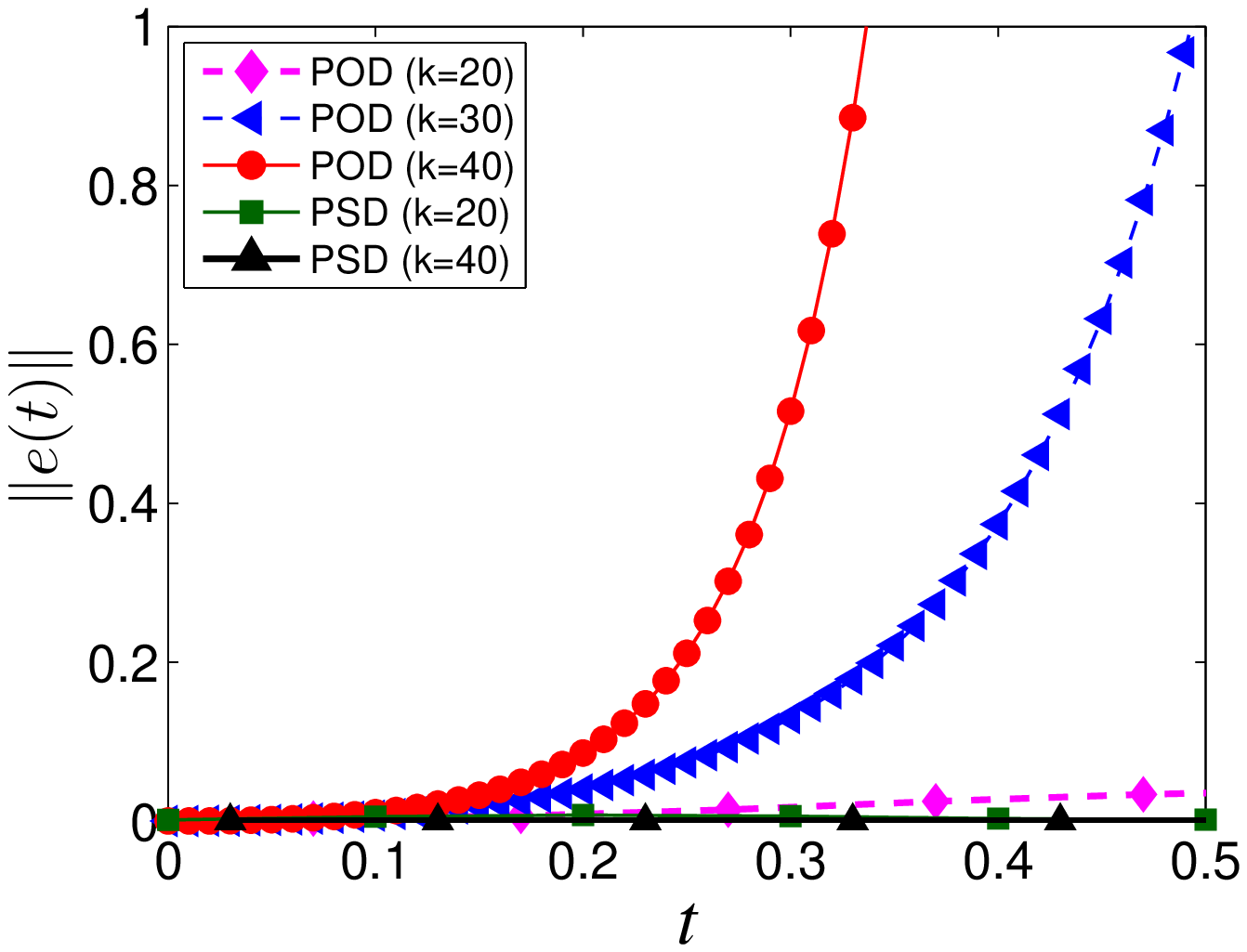}
\end{center} \end{minipage}\\
\begin{minipage}{0.48\linewidth}\begin{center} (a) \end{center}\end{minipage}
\begin{minipage}{0.48\linewidth}\begin{center} (b) \end{center}\end{minipage}
 \caption{(Color online.)   Comparison between POD and PSD (cotangent lift) reduced systems of the dissipative wave equation.  (a) The evolution of the $L^2$ error norm, $\|e(t)\|:=\|\hat u(t)-u(t)\|$, between  the benchmark solution $u(t)$ and approximating solutions $\hat u(t)$   for the time domain $t\in[0, 50]$. (b) The  $L^2$ error norm $\|e(t)\|$ for the zoomed in time interval  $t\in[0, 0.5]$.
} \label{fig:waveerr}
 \vspace{-3mm}
\end{center}
\end{figure}

\begin{figure}
\begin{center}
\begin{minipage}{0.48\linewidth} \begin{center}
\includegraphics[width=1\linewidth]{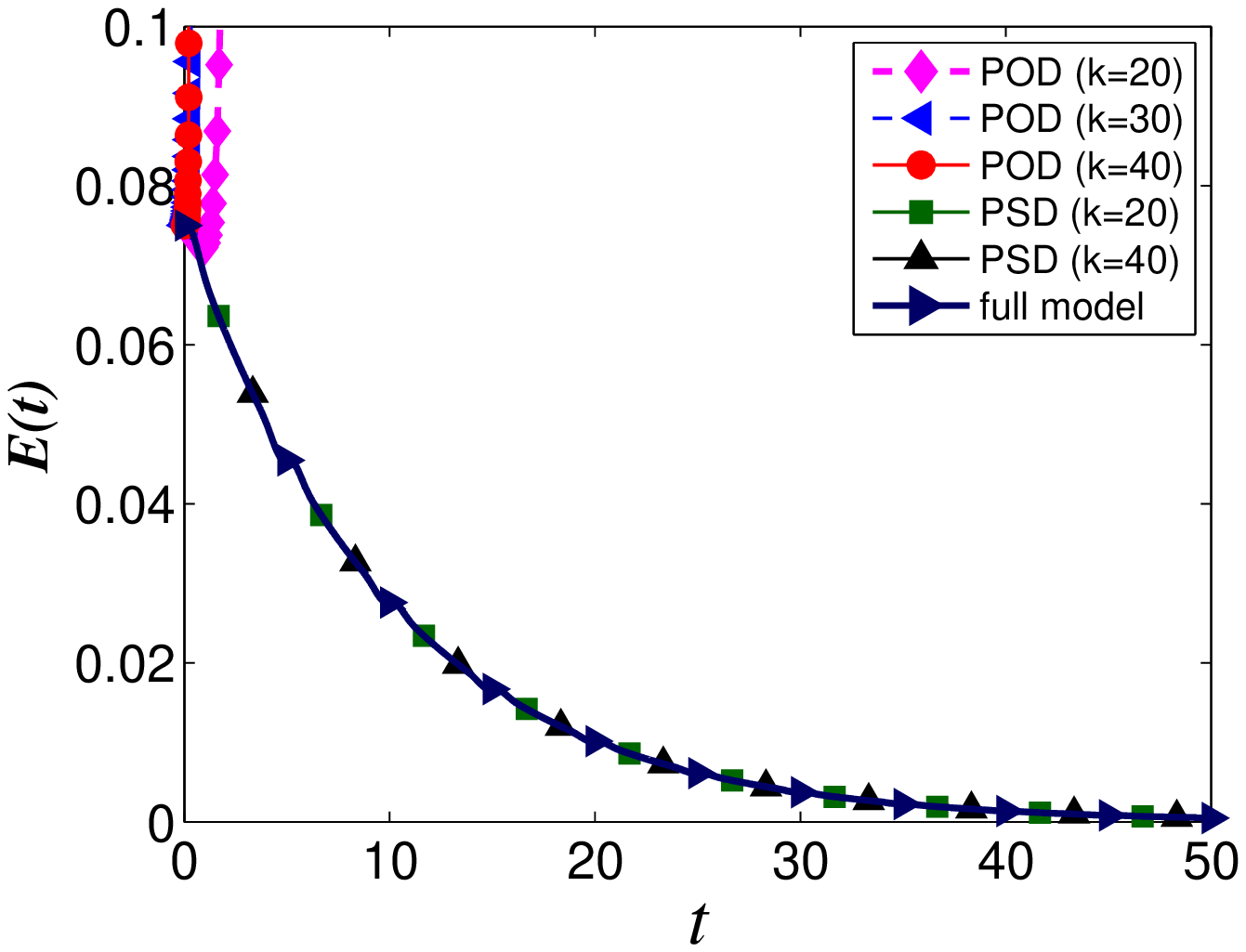}
\end{center} \end{minipage}
\begin{minipage}{0.48\linewidth} \begin{center}
\includegraphics[width=1\linewidth]{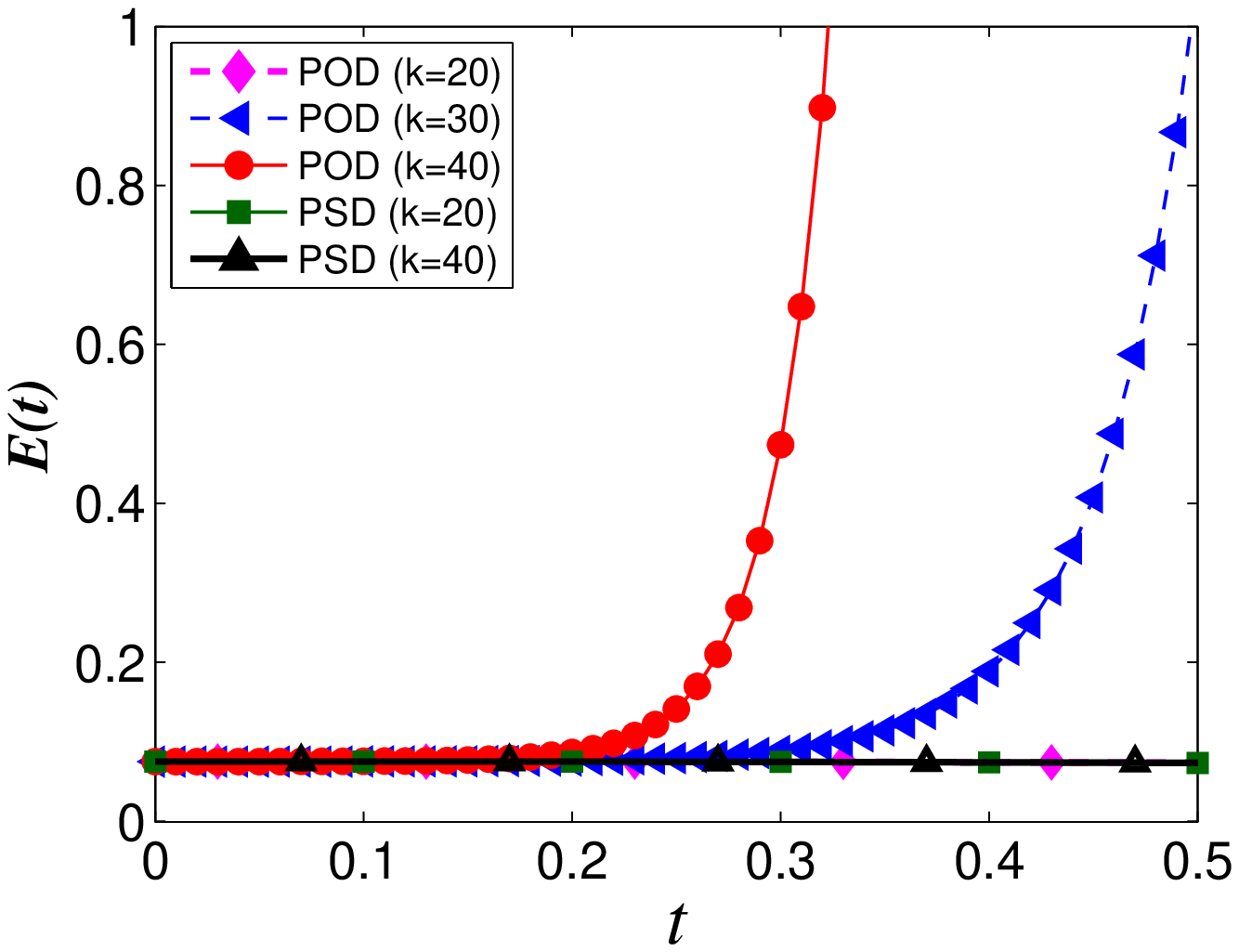}
\end{center} \end{minipage}\\
\begin{minipage}{0.48\linewidth}\begin{center} (a) \end{center}\end{minipage}
\begin{minipage}{0.48\linewidth}\begin{center} (b) \end{center}\end{minipage}
 \caption{(Color online.)  Comparison between full model,  POD reduced system, and PSD (cotangent lift) reduced system of the dissipative wave equation.
(a) The evolution of the system energy $E(t)$  for the time domain $t\in[0,50]$. (b) The  system energy $E(t)$ for the zoomed in time interval $t\in[0, 0.5]$.
} \label{fig:waveeng}
 \vspace{-3mm}
\end{center}
\end{figure}

\begin{figure}
\begin{center}
\begin{minipage}{0.48\linewidth} \begin{center}
\includegraphics[width=1\linewidth]{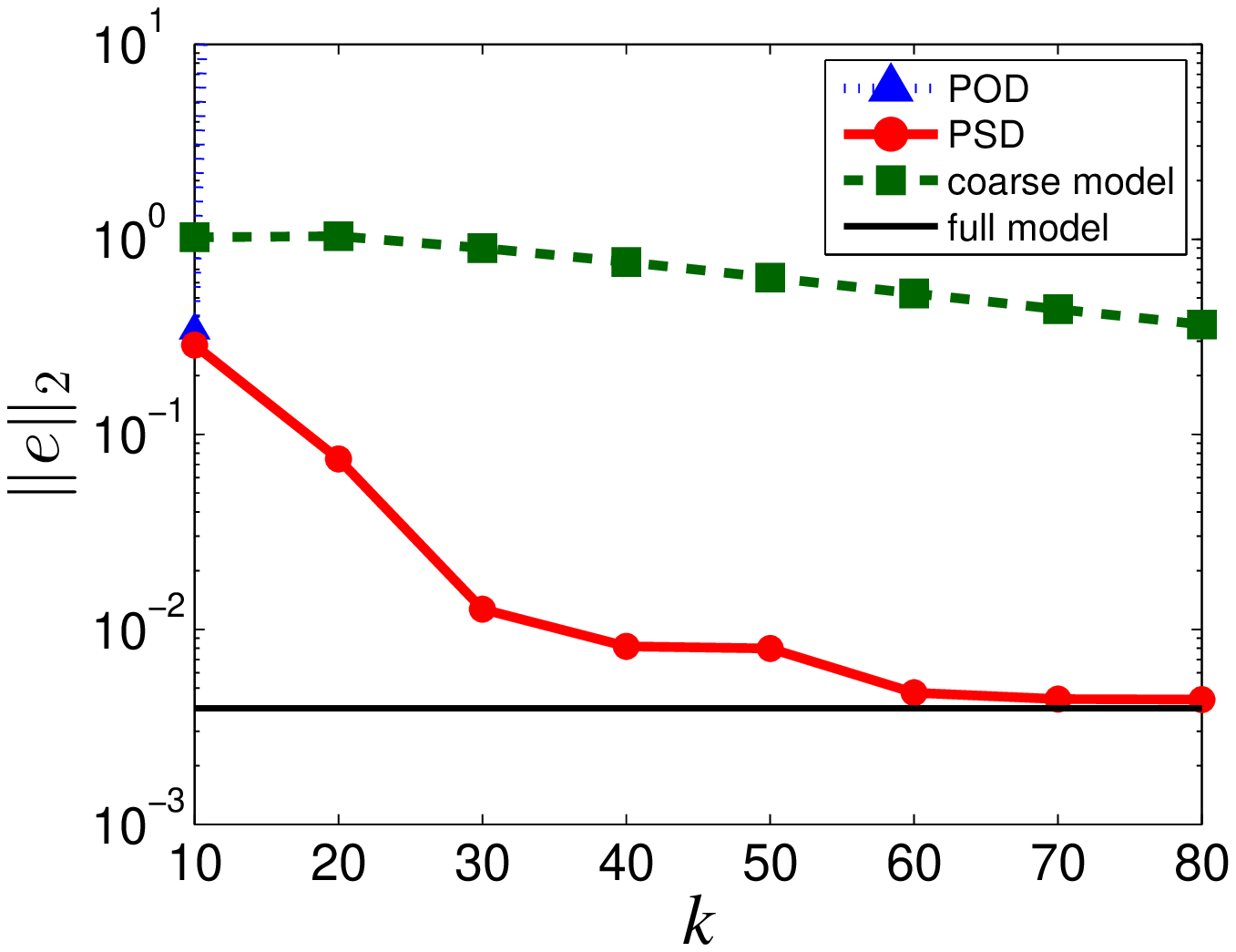}
\end{center} \end{minipage}
\begin{minipage}{0.48\linewidth} \begin{center}
\includegraphics[width=1\linewidth]{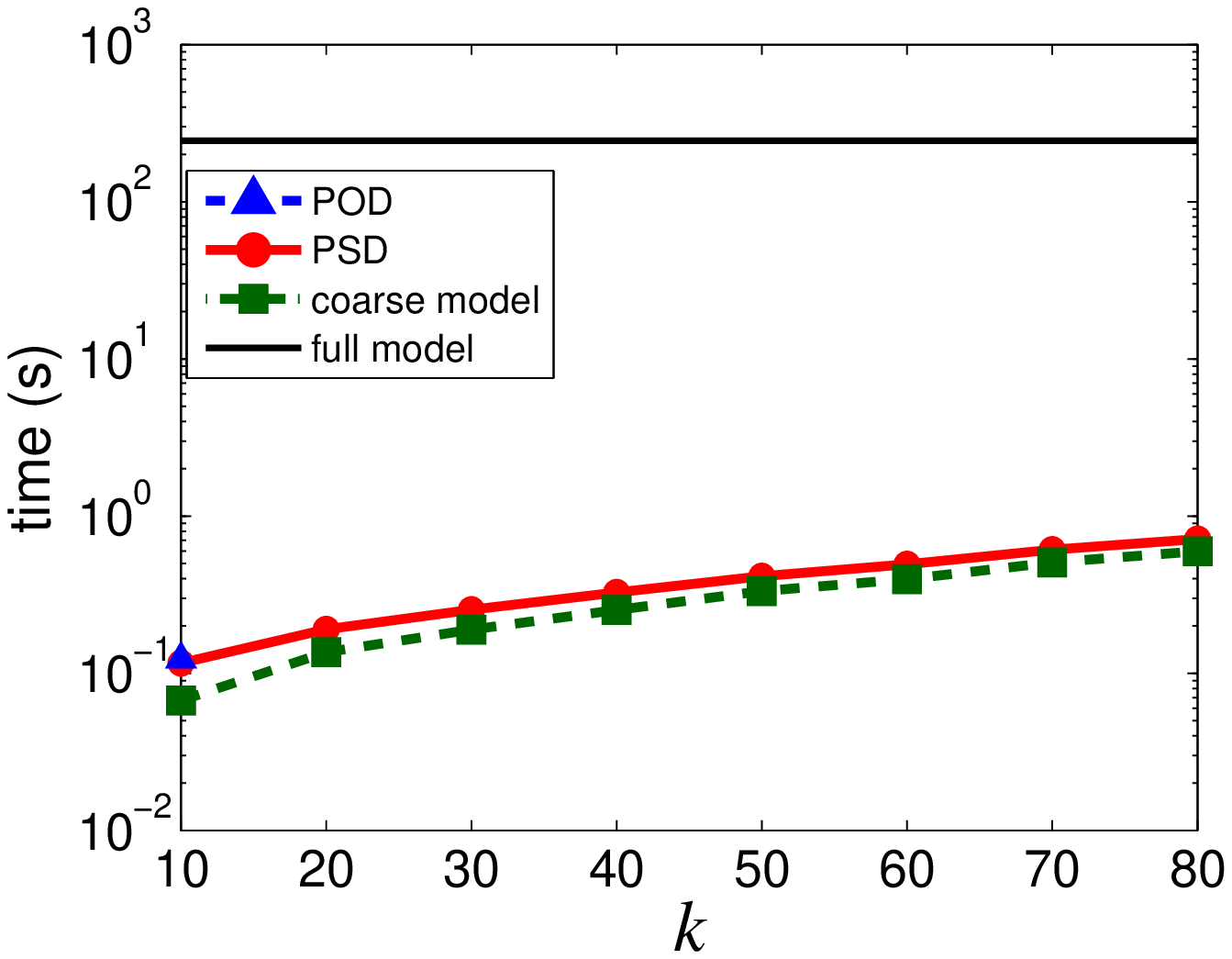}
\end{center} \end{minipage}\\
\begin{minipage}{0.48\linewidth}\begin{center} (a) \end{center}\end{minipage}
\begin{minipage}{0.48\linewidth}\begin{center} (b) \end{center}\end{minipage}
 \caption{(Color online.) Comparison between the full model, coarse model,  POD  reduced system, and PSD (cotangent lift) reduced system for the dissipative wave equation.
  (a) The $L^2$ norm of the total error $\|e\|_2:=\sqrt {\int_0^T {{{\left\| {e(t)} \right\|}^2}{\rm{d}}t} }$ of different systems. For the POD reduced system, we only compute $\|e\|_2$ with 10 modes; when the subspace dimension $k$ is greater than 10, the POD reduced system blows up in the whole time domain [0, 50] and $\|e\|_2$ becomes infinite. (b) The running time of different systems corresponding to different subspace dimensions $k$. All data comes from the average value of 10 independent runs.
} \label{fig:waveerr2}
 \vspace{-3mm}
\end{center}
\end{figure}

However, preserving the data does not necessarily imply preserving the dynamics. With more modes, there is no guarantee that the POD reduced system will yield more accurate solutions. As Figure \ref{fig:waveerr} indicates, the $L^2$ error norm of the POD reduced system increases exponentially when it has 20, 30, or 40 modes. In addition, the POD reduced system with 40 modes blows up faster than the POD system with 20 modes. This result verifies that POD can yield unstable reduced systems, even though the original system is dissipative and stable.  By contrast, PSD reduced systems have small numerical errors in the $L^2$ norm for all the tested cases. Figure \ref{fig:waveeng} shows that PSD reduced models accurately capture  the evolution of the system energy $E$ of the dissipative wave equation, while the energy of POD reduced systems quickly grows to infinity. In this example, increasing the number of POD modes actually causes the system energy to increase at a faster rate. Here, $E$ equals the discretized Hamiltonian $H_d(y)$.

Figure \ref{fig:waveerr2}(a) plots the $L^2$ norm of  the total error of different systems over the whole time domain [0, 50].  We compare the full model (with $k=1000$), coarse model, as well as POD and PSD reduced model. The subspace dimension $k$ of the coarse model and reduced models ranges from 10 to 80. The $L^2$ norm of  the total error of the POD reduced system is bounded  only when $k$ equals 10 for all the tested cases .  While the PSD reduced system show some numerical error, this error quickly converges to the error of the full model. The coarse model also preserves the forced Hamiltonian structure and remains stable, but the numerical error of the coarse model reduces at a low rate with increased modes. Figure \ref{fig:waveerr2}(b) shows the running time of different methods. We find coarse model and POD/PSD reduced model have similar running speed. With 80 modes,  both the coarse model and the reduced model can significantly improve the computational efficiency and reduce the running time by more than two orders of magnitude.

\subsection{Stability analysis}
Using the numerical results, we further analyze the stability for the linear system in  (\ref{lwave}). Using (\ref{waveeng}), we know $\mathop {\lim}\limits_{x \to \infty}  H_d(x)=+\infty$; Lemma \ref{cor:S} implies that the full model is uniformly bounded. Since the origin is the strict minimum of $H_d$, Theorem \ref{minimal} implies that the origin is a stable equilibrium for  the dissipative wave equation. Since the external force of (\ref{lwave}) is a Rayleigh dissipative force, where $\mathcal{F}(q, \dot q)=\frac{1}{2} \dot q^T \dot q$, the reduced PSD system is also dissipative. By the same argument, the reduced PSD system is uniformly bounded, and also has the origin as a stable equilibrium.

To explain why the POD reduced system is unstable, we study the eigenvalues of the linear wave equation. According to \cite{LeVequeRJ:07a}, the eigenvalues $\beta _i$ ($i= 1, \ldots, n$) of the discretized spatial derivative ${D}_{xx}$ with periodic boundary conditions are given by
\[ \beta _i = -\frac{2}{\Delta x^2} \left[   1- \cos \left( \frac{2\pi i}{n} \right) \right].\]
It follows that the eigenvalues of the full model $K+L$ in (\ref{expansion}) are given by $2n$ complex numbers $\{\lambda_i\}_{i=1}^{2n}$, where $\lambda_i, \lambda_{i+n}$ are solutions of $\lambda^2+\beta \lambda -c^2 \beta_i+\omega_0^2=0$ for $i=1, \ldots, n$. It can be verified that all the eigenvalues of the full model have negative real parts, which means the full model is stable.

Since POD does not preserve the system energy, there are no mechanisms to confine the solution in a bounded region.  As a result, the reduced system may blow up with time evolution.  To corroborate this claim, let $\Phi$ denote a POD basis matrix, $\lambda_*$ denote the eigenvalue of $\Phi^T (K+L) \Phi$ with the maximal real part, and $\xi_*$ denote the corresponding eigenvector with unit length.  Then, $a_*=\xi_*^T y_0$ gives the projection coefficient of $y_0$ onto $\xi_*$. Since the solution of a linear system has an exponential term $a_* \exp(\lambda_* t)\xi_*$, the POD reduced system is unstable when $a_*\ne 0$ and $\rm{Re}(\lambda_*)>0$.

Table \ref{tab:podeigv} lists $\rm{Re}(\lambda_*)$ with a wide range of  diffusion coefficients $\beta$ and subspace dimensions $k$. Numerical results show that $a_*\ne 0$ for all the tested cases. When $\beta=10^2$, the diffusion term becomes dominant in (\ref{expansion}). The POD reduced system is stable when $k=10$ and $k=20$ for the tested cases. When $10^{-2}\le \beta\le 10^1$, The POD reduced system is stable only when $k=10$.  When $\beta=10^{-3}$, the diffusion term becomes negligible in (\ref{expansion}) and the POD reduced system is unstable for all the tested cases.
Table \ref{tab:podeigv} also shows that when $\beta=10^{-1}$, ${\rm{Re}}(\lambda_*)$ with 40 modes is much larger than ${\rm{Re}}(\lambda_*)$ with 20 modes, which explains why the POD reduced system with 40 modes blows up faster than the system with 20 modes in Figure \ref{fig:waveerr}.

\begin{table} [htbp]
\begin{center}
\caption{The real part $\rm{Re}(\lambda_*)$ of the eigenvalue corresponding to the most unstable POD mode  for different diffusion coefficients $\beta$ and subspace dimensions $k$.}
 \label{tab:podeigv}
\begin{tabular}{c|cccccccc}
& \multicolumn{8}{c}{$k$} \\
\cline{2-9}
 $\beta$ &  10&	20&	30&	40&	50&	60&	70 &80 \\
\hline
$10^{-3}$&  $5.17 \times 10^{-3}$&	0.304&	15.1&	19.9&	16.7&	17.4&	19.6& 111\\
$10^{-2}$&  $-3.71 \times 10^{-3}$& 0.252 &	15.6&	20.2&	17.0&	17.9&	19.7& 113\\
$10^{-1}$&  $-2.09 \times 10^{-2}$&	1.26&	12.3&	18.0&	21.8&	20.5&	22.3& 129\\
$10^0$&     $-2.51 \times 10^{-3}$&	1.43&	31.4&	37.2&	26.3&	60.3&	44.6& 139\\
$10^1$&     $-2.50 \times 10^{-4}$&	1.08&	18.7&	26.7&	38.0&	43.1&	47.8& 50.0\\
$10^2$&     $-1.07 \times 10^{-3}$&	$-2.50 \times 10^{-5}$&	3.66&	32.8&	33.6& 43.8&	54.7& 64.2\\
\end{tabular}
\end{center}
\end{table}

\section{Conclusion}\label{sec:conclusion}
This paper proposed a PSD model reduction method to simplify large-scale forced Hamiltonian systems, which can achieve significant computational savings. Since the PSD reduced system preserves the forced Hamiltonian structure, it automatically satisfies the d'Alembert's principle. Since  d'Alembert's principle is the first principle  in classical mechanics, the PSD reduced system is a physical model, rather than merely a numerical model. In contrast, although POD can always reduce the dimensionality of a dynamical system, a POD reduced system may be or may not be physical, since there is no guarantee that the system can satisfy any fundamental physical laws.

Two structure-preserving approaches are developed in order to reconstruct reduced systems in a low-dimensional subspace, one based on the variational principle and the other on the structure-preserving projection. Both approaches can yield the same structure-preserving reduced system. By incorporating the vector field into the data ensemble, the PSD method also preserve the time derivative of system energy.
In a special case when the external force represents the Rayleigh dissipation,  PSD  automatically preserves the dissipativity. As a consequence, PSD also preserves the boundedness and Lyapunov stability under some conditions.

The stability, accuracy, and efficiency of the proposed method are illustrated through numerical simulations of the one dimensional dissipative  wave equation. However, PSD can have much more general applications. Once we choose canonical coordinates, all the systems that satisfy d'Alembert's principle can be written as  the forced Hamiltonian equation. As a result, PSD can be applied to any large-scale mechanical  system in principle. Finally, we should mention that the computational complexity and implementation complexity of PSD are almost identical to the complexity of POD. Since  the POD reduced system can be unstable and produces unpredictable results, we believe that PSD is more suited for model reduction of large-scale mechanical systems, especially when long-time integration is required.

\bibliographystyle{siam}
\bibliography{RefA2}

\end{document}